\documentclass[a4paper, reqno, 11pt]{amsart}

\usepackage{enumitem}
\usepackage[utf8]{inputenc}
\usepackage[english]{babel}
\usepackage{amssymb,amsmath,amsthm}
\usepackage{bbm}
\usepackage{framed}
\usepackage{fancyhdr}
\usepackage{enumitem}
\usepackage{ifthen}
\usepackage{color}
\usepackage{hyperref}
\usepackage{graphicx}
\usepackage{csquotes}
\usepackage{cleveref}
\usepackage{appendix}

\theoremstyle{plain} 
\theoremstyle{plain} \newtheorem{theorem}{Theorem}
\theoremstyle{plain} \newtheorem{corollary}{Corollary}
\theoremstyle{plain} 
\theoremstyle{plain} \newtheorem{lemma}{Lemma}
\theoremstyle{plain} 
\theoremstyle{plain} \newtheorem{proposition}{Proposition}
\theoremstyle{plain} \newtheorem{definition}{Definition}
\theoremstyle{definition}
\theoremstyle{definition}\newtheorem{remark}{Remark}
\theoremstyle{definition}
\theoremstyle{definition} 

\newcommand{\Rdd}{{\mathbb{R}^{N}}}

\newcommand{\R}{{\mathbb{R}}}

\newcommand{\eps}{\varepsilon}

\newcommand{\nphi}{\mathcal{B}\varphi}

\newcommand{\tW}{\tilde{W}}

\newcommand{\al}{\alpha}
\newcommand{\lm}{\lambda}

\newcommand{\vsp}[1]{\vspace{0.5cm}\par}

\newcommand{\mres}{\mathbin{\vrule height 1.6ex depth 0pt width
0.13ex\vrule height 0.13ex depth 0pt width 1.3ex}}
\newcommand{\margnote}[1]{
\ifthenelse{\boolean{shownotes}}%
{\marginpar{\raggedright\tiny\texttt{#1}}}%
{}%
}
\newcommand{\hole}[1]{
\ifthenelse{\boolean{shownotes}}%
{\begin{center} \fbox{ \rule {.25cm}{0cm}
\rule[-.1cm]{0cm}{.4cm} \parbox{.85\textwidth}{\begin{center}
\texttt{#1}\end{center}} \rule {.25cm}{0cm}}\end{center}}
{}
}

\numberwithin{equation}{section}

\begin{document}

\title[$\mathcal{A}$-quasiconvexity, G\aa rding inequalities and applications]{$\mathcal{A}$-quasiconvexity, G\aa rding inequalities and applications in PDE constrained problems in dynamics and statics}

\author[K. Koumatos]{Konstantinos Koumatos}
\address[Konstantinos Koumatos]{\newline
Department of Mathematics, University of Sussex\\ Pevensey 2 Building,
Falmer,
Brighton, BN1 9QH, UK}
\email[]{\href{k.koumatos@sussex.ac.uk}{k.koumatos@sussex.ac.uk}}

\author[A. Vikelis]{Andreas Vikelis}
\address[Andreas Vikelis]{\newline
Department of Mathematics, University of Sussex\\ Pevensey 2 Building,
Falmer,
Brighton, BN1 9QH, UK}
\email[]{\href{a.vikelis@sussex.ac.uk}{a.vikelis@sussex.ac.uk}}

\begin{abstract}
 A G\aa rding-type inequality is proved for a quadratic form associated to  $\mathcal{A}$-quasiconvex functions. This quadratic form appears as the relative entropy in the theory of conservation laws and it is related to the Weierstrass excess function in the calculus of variations. The former provides weak-strong uniqueness results, whereas the latter has been used to provide sufficiency theorems for local minimisers. Using this new G\aa rding inequality we provide an extension of these results to PDE constrained problems in dynamics and statics under $\mathcal{A}$-quasiconvexity assumptions. The application in statics improves existing results by proving uniqueness of $L^p$ local minimisers in the classical $\mathcal{A}={\rm curl}$ case.
 \vspace{0.1cm}

 \vspace{0.1cm}

\end{abstract}

\maketitle

\section{Introduction}

In the setting of continuum mechanics and the theory of electromagnetism, often problems are constrained by linear partial differential equations (PDEs), that is their solutions are constrained to lie in the kernel of a certain differential operator $\mathcal{A}$. The prototypical example arises in elasticity. In elastostatics, one is concerned with the minimisation of the functional
\begin{equation}\label{eq:intro0}
\int_\Omega W(\nabla y)
\end{equation}
and thus solutions $U = \nabla y$ are constrained by the operator $\mathcal{A} = {\rm curl}$. Similarly, in dynamics, the equations of elasticity can be written in the form of the first-order system of conservation laws
\begin{align*}
   	\partial_t v - {\rm div} DW(F)&=0, \\
   	\partial_t F - \nabla v&=0,\\
   	{\rm curl}\,F&=0.
\end{align*}
Note that the last equation constrains solutions $F$ to be gradients and it is satisfied as long as the initial data are curl-free. More generally, one may consider problems constrained by other differential operators $\mathcal{A}$, leading to the study of minimisation problems of the form
\begin{equation}\label{eq:intro1}
\mathcal{W}(U) = \int W(U),\quad\mathcal{A}U = 0
\end{equation}
and systems of conservation laws
\begin{equation}\label{eq:intro2}
\partial_t U + {\rm div} f(U) = 0
\end{equation}
with involutions $\mathcal{A}$, i.e. with the property that $\mathcal{A}U(t,\cdot) = 0$ whenever $\mathcal{A}U(0,\cdot) = 0$. PDE constrained problems of the above form, and others, have been studied extensively. Indeed, the theory of compensated compactness developed by Murat and Tartar originated within this $\mathcal{A}$-free context \cite{murat1, murat2, murat3}. 

In particular, they understood that quadratic forms that are convex in certain directions associated to $\mathcal{A}$ are lower semicontinuous along $\mathcal{A}$-free, weakly converging sequences. This set of directions, $\Lambda_{\mathcal{A}}$, is referred to as the wave cone of $\mathcal{A}$, see Section \ref{sec:prel}, and contains the amplitudes along which ellipticity is lost. For example, for vectorial problems and $\mathcal{A} = {\rm curl}$, the wave cone consists of rank-one matrices and rank-one convexity becomes the relevant convexity condition. Note that rank-one convexity for quadratic forms is equivalent to the less transparent notion of quasiconvexity which is itself equivalent to the weak lower semicontinuity of \eqref{eq:intro0}, see \cite{Dacorogna}.

Indeed, for problems of the form \eqref{eq:intro1}, an appropriate extension of quasiconvexity, called $\mathcal{A}$-quasiconvexity, was introduced by Dacorogna \cite{DacQC} and shown to be equivalent to the weak lower semicontinuity of \eqref{eq:intro1}, in \cite{DacQC, Fon}. 
More recently, and following the work in \cite{Guido}, a developing body of literature has emerged on PDE constrained problems, including results on appropriate modifications of BV spaces, lower semicontinuity, Young measures, Sobolev-type inequalities,  and others \cite{Arroyo19,  Anna2, breit19, BogEmb, BogKris, Potentials, Anna1, skipper19}.

In the context of dynamics, Dafermos in \cite{Daf} studied the system of conservation laws \eqref{eq:intro2} endowed with involutions where $\mathcal{A} = \sum_\alpha A_\alpha \partial_\alpha$ was assumed to be a first order operator. 
He showed that if the involutions are complete (see \cite{Daf13}) system \eqref{eq:intro2} becomes hyperbolic and constructed a first order potential operator $\mathcal{B}= \sum_\alpha B_\alpha\partial_\alpha$ such that $U=\mathcal{B}W$ whenever $\mathcal{A}U=0$. Through this potential $\mathcal{B}$, he extracted a Poincar\'e type inequality for $\mathcal{A}$-free functions which played a decisive role in the proof of his main tool: a G\aa rding-type inequality for the quadratic form 
\begin{align}\label{eq:intro_quadratic_form}
\eta(U|\bar U) & = \eta(U) - \eta(\bar U) - D\eta(\bar U)\cdot (U - \bar U)\nonumber \\
&= \int_0^1(1-t)D^2\eta\left(\bar U + t(U - \bar U)\right)\,dt (U - \bar U)\cdot (U - \bar U),
\end{align}
associated to the $\Lambda_\mathcal{A}$-convex entropy $\eta$. Nevertheless, this G\aa rding inequality required that the weak solutions, assumed bounded and in the space $BV$, satisfy an extra assumption of small local oscillations. Then, naturally, it leads to stability and weak-strong uniqueness results for such entropic weak solutions. In \cite{Spirito} and the case of elasticity, it was understood that the crucial G\aa rding inequality and the subsequent weak-strong uniqueness result can be proved without the assumption of small oscillations, provided the entropy instead satisfies the stronger condition of quasiconvexity\footnote{J. Kristensen and J. Campos Cordero \cite{KrisJCC} have obtained a similar G\aa rding inequality in the ${\rm curl}$-free setting following a different approach.}.

More generally, G\aa rding inequalities have been very important, for example, to establish existence, uniqueness and regularity for elliptic problems, see \cite{vich64, garding53, BVmin, mullerNotes, necas67, sil17}. Crucially, a G\aa rding-type inequality for the quadratic form in \eqref{eq:intro_quadratic_form} also appeared in the resolution of the so-called Weierstrass problem in the vectorial calculus of variations, i.e. the problem of finding (quasiconvexity based) sufficient conditions for a map $\bar y$ to be a strong (or $L^p$) local minimiser of \eqref{eq:intro0}, see Section \ref{sec:min}. This was indirectly employed in the original proof of \cite{GM09} and more explicitly in the subsequent proofs in \cite{JudithR, JudithKostas} which seek the positivity of
\[
\int W(\nabla \bar y + \nabla\varphi|\nabla \bar y),
\]
related to a G\aa rding-type inequality for the quadratic form $W(\cdot|\cdot)$. In this context, this quadratic form is known as the Weierstrass excess or E-function, see \cite{JudithKostas, GM09} for functionals depending on lower order terms.

In the present work, we consider general constant coefficient, linear differential operators $\mathcal{A}$ with constant rank and, for $p\geq 2$, prove the G\aa rding inequality, see Theorem \ref{theorem:2}, 
\begin{equation}\label{eq:intro_garding}
   \int_Q W(\bar U + \psi|\bar U) \gtrsim \int_Q \left(|\psi|^2 + |\psi|^p\right) - \|\psi\|_{W^{-1,2}}^2 - \|\psi\|_{W^{-1,p}}^p,
\end{equation}
for $\mathcal{A}$-quasiconvex $W$, $\bar U\in C^0(Q)$, and $\psi \in L^p(Q)$, $\mathcal{A}$-free and zero-average on the unit torus $Q$. This is the content of Section \ref{sec:garding} where we also prove Lemma \ref{Dec}, an extension of the Decomposition Lemma in \cite{JudithR}, see also \cite{KrisDec}, in the $\mathcal{A}$-free setting. 

Moreover, in Section \ref{sec:involutions}, we employ the G\aa rding inequality to prove stability and weak-strong uniqueness results for dissipative solutions of conservation laws with involutions under the assumption that the entropy is $\mathcal{A}$-quasiconvex. We note that no restrictions on the order of $\mathcal{A}$ are required and as in \cite{Daf} weak solutions need only be bounded but with no additional assumptions on the local oscillations. 
In Section \ref{sec:min}, we study a generalisation of the Weierstrass theory in this $L^p$, $\mathcal{A}$-free setting, see Theorem \ref{thm:min}, which comes naturally from the proof of Theorem \ref{theorem:2}. We note that our result entails the uniqueness of local minimisers in a quantitative way. This adds to the existing results in the case $\mathcal{A} = {\rm curl}$ and bounded domains. Indeed, in Corollary \ref{cor:unique}, we prove this uniqueness result for $\mathcal{A} = {\rm curl}$, bounded domains, and free boundary conditions, under the additional assumption of quasiconvexity at the boundary as in \cite{ JudithR, JudithKostas, GM09}. Section \ref{sec:prel}, collects all essential definitions and known results in the $\mathcal{A}$-free setting that are used in the following sections. 

\section{Preliminaries}\label{sec:prel}

\subsection{Constant Rank Linear Operators}\label{sb11}
 
For each d-multi index $\alpha$, let us consider a collection of linear operators $A_{\alpha}\in {\rm Lin}(\mathbb{R}^N,\mathbb{R}^M)$. We define a homogeneous k-th order linear operator $\mathcal{A}$ by
\begin{equation}\label{eq:A}
\mathcal{A}\psi:=\sum_{|\alpha|=k}A_{\alpha}\partial^{\alpha}\psi,\,\, \psi:Q\subseteq\mathbb{R}^d\to \mathbb{R}^N,
\end{equation}
where $Q=(0,1)^d$ is identified with the unit torus $\mathbb{T}^d$. We think of $\mathcal{A}$ as a polynomial in $\partial$ and so we write its principal symbol as
\[
\mathbb{A}:\mathbb{R}^d\to {\rm Lin}(\mathbb{R}^N,\mathbb{R}^M),\,\,\,\,\,\,\, \mathbb{A}(\xi)=(2\pi i)^k\sum_{|\alpha|=k}A_{\alpha}\xi^{\alpha}.
\]
The wave cone associated with $\mathcal{A}$ is denoted by 
\[
\Lambda_\mathcal{A}=\bigcup_{\quad\xi\in S^{d-1}}{\rm ker}\,\mathbb{A}(\xi),
\]
and contains the amplitudes $\lambda\in\mathbb{R}^N$ along which the system fails to be elliptic where ellipticity means that ${\rm ker}\,\mathbb{A}(\xi) = \{0\}$ for all $\xi\neq 0$. Indeed, $\lambda\in\Lambda_{\mathcal{A}}$ if and only if the operator $\mathcal{R}_\lambda(v):=\mathcal{A}(\lambda v)$ is not elliptic, where $v\in C^{\infty}(Q;\mathbb{R})$. Moreover, we assume throughout that the linear differential operator $\mathcal{A}$ has the constant rank property, i.e. there exists $r\in\mathbb{N}$ such that 
\[
{\rm rank}\;\mathbb{A}(\xi)=r \text{ for all }\xi\in S^{d-1}.
\]
The constant rank assumption, first introduced in the context of compensated compactness by Murat \cite{murat1}, ensures the smoothness of the projection mapping
\[\mathbb{P}:\mathbb{R}^d\setminus\{0\}\to {\rm Lin}(\mathbb{R}^N,\mathbb{R}^N),\,\,\,\xi\mapsto {\rm Proj}_{{\rm ker}\mathbb{A}(\xi)},
\]
and thus makes tools of pseudo-differential calculus available. Using some of these tools together with a result of Decell \cite{decell}, Rai{\c{t}}{\u{a}} in \cite{Potentials} gave a new characterisation for constant rank operators:
\begin{theorem}\label{thm:bogpot}
Let $\mathcal{A}$ be a linear homogeneous differential operator with constant coefficients. Then $\mathcal{A}$ has constant rank if and only if there exists a linear homogeneous differential operator $\mathcal{B}$ with constant coefficients such that 
\[
{\rm im}\,\mathbb{B}(\xi)= {\rm ker}\,\mathbb{A}(\xi) \text{ for 
 all }\xi\in\mathbb{R}^d\setminus\{0\}.
\]
We write, for some $B_{\alpha}\in {\rm Lin}(\mathbb{R}^{M'},\mathbb{R}^N)$,
\begin{equation}\label{eq:B}
\mathcal{B}\varphi:=\sum_{|\alpha|=l}B_{\alpha}\partial^{\alpha}\varphi,\,\,\, \varphi:Q\subseteq\mathbb{R}^d\to \mathbb{R}^{M'}.
\end{equation}
\end{theorem}

We refer to the potential operator $\mathcal{B}$ simply as the potential of $\mathcal{A}$ although no meaningful notion of uniqueness is known, see \cite{NL} for a discussion. Lemma 3 in \cite{Potentials} implies also that the operator $\mathcal{B}$ has constant rank. 

\subsection{Sobolev estimates}

 Henceforth, for a function $\psi\in L^p(Q)$ we say that ``$\mathcal{A}\psi=0$ in $Q$" in the sense of distributions on the torus, i.e.
\begin{equation}\label{eqDC}
-\int_{Q}\psi\cdot\mathcal{A}^*\phi=0\,\,\,\text{for all}\,\,\,\phi\in C^\infty(Q), 
\end{equation}
where $C^\infty(Q)$ consists of smooth, $Q$-periodic functions and $\mathcal{A}^*$ is the adjoint operator. We call $\mathcal{A}$-free any function satisfying \eqref{eqDC}.

In this section, we present some fundamental estimates in Sobolev spaces for a class of primitive functions which we refer to as $\mathbb{B}^\dagger$-primitives, constructed in \cite{Potentials}. These estimates are necessary to replace Poincar\'e-type inequalities which we particularly require when introducing cut-offs. We note that these estimates may fail for 
general primitives.

\begin{remark}\label{remark1}
Throughout, $W^{l,q}(Q)$ denotes the closure of $C^\infty(Q)$ in the $W^{l,q}$ norm.
The space $W^{-l,p}(Q)$ is its dual and its norm is equivalent to 
\[
\Big\|\mathcal{F}^{-1}\Big[ \frac{\hat{\psi}(\xi)}{(1+|\xi|^2)^{l/2}}\Big] \Big\|_{L^p(Q)}.
\]
Note that when $\int_Q u=0$ this norm is equivalent to the norm
\[
\Big\|\mathcal{F}^{-1}\Big[ \frac{\hat{\psi}(\xi)}{|\xi|^l} \Big] \Big\|_{L^p(Q)}
\]
since the $Fourier$ multipliers $(1+|\xi|^2)^{-l/2}$ and $|\xi|^{-l}$ are comparable for $\xi\in\mathbb{Z}^d\setminus\{0\}$.
\end{remark}

\begin{lemma} \label{Ptype}
Let $\mathcal{A}$ and $\mathcal{B}$ as in Theorem \ref{thm:bogpot}. Then for all $\mathcal{A}$-free functions $\psi\in L^p(Q)$ with $\int_Q\psi=0$, there exists $\varphi\in W^{l,p}(Q)$ such that

\begin{enumerate}[label=(\roman*)]
    \item $\psi=\mathcal{B}\varphi$ ;
    \item $\|\varphi\|_{L^p(Q)}\leq C \|\psi\|_{W^{-l,p}(Q)}$ ;
    \item $\|\varphi\|_{W^{l,p}(Q)}\leq C \|\psi\|_{L^p(Q)}$;
    \item $\|\varphi\|_{W^{l-i,p}(Q)}\leq C \|\psi\|_{W^{-1,p}(Q)}$ for all $i=1,..,l$.
\end{enumerate}

We will call $\varphi$ the $\mathbb{B}^\dagger$-primitive of $\psi$.
\end{lemma}

Although $(ii)$, $(iii)$ and $(iv)$ follow from the construction in \cite{Potentials}, a proof is not explicitly given. Hence, for completeness, we provide a proof here.

\begin{proof}
 We prove the result for $\psi\in C^{\infty}(Q)$ and the general case follows by approximation. 
 Indeed, $(i)$ is known from \cite[Lemma 2]{Potentials}, where the primitive function $\varphi\in C^\infty(Q)$ is constructed as 
 \[
 \varphi(x)=\sum_{\xi\neq 0}\mathbb{B}^\dagger(\xi) \widehat{\psi}(\xi) e^{2\pi i\xi\cdot x},
 \]
and $\mathbb{B}^\dagger(\cdot)$ is the pseudo?inverse of $\mathbb{B}(\cdot)$ which is itself smooth whenever $\mathcal{B}$ is, see \cite{NL}. This justifies our adopted terminology $\mathbb{B}^\dagger$-primitive.

For $(ii)$, since $\mathbb{B}^\dagger(\cdot)$ is smooth and $(-l)$-homogeneous, the operator $\mathbb{B}^\dagger(\xi/|\xi|)$ is $0$-homogeneous and smooth, and thus a Fourier multiplier, see \cite[Proposition 2.13]{Fon}. Hence, by the Mikhlin multiplier theorem and Remark \ref{remark1}
\begin{align*}
\|\varphi\|_{L^p(Q)}=\Big\|\mathcal{F}^{-1}\Big[ \frac{1}{|\xi|^l}\mathbb{B}^\dagger(\frac{\xi}{|\xi|})\widehat{\psi}(\xi) \Big] &\Big\|_{L^p}
\lesssim \Big\|\mathcal{F}^{-1}\Big[ \frac{1}{|\xi|^l}\widehat{\psi}(\xi) \Big] \Big\|_{L^p}=\|\psi\|_{W^{-l,p}(Q)},
\end{align*}

For $(iii)$, by applying the Poincar\'e inequality for all the derivatives of $\varphi$, since $\int_Q \nabla^i\varphi=\widehat{\nabla^i\varphi}(0)=0$, we have that $\|\nabla^{l-i}\varphi\|_{L^p}\lesssim \|\nabla^l\varphi\|_{L^p}$ for all $i$=0,..,$l$ and so $\|\varphi\|_{W^{l,p}}\lesssim \|\nabla^l\varphi\|_{L^p}$.
Then, by differentiating $\varphi$ we obtain 
\[
\nabla^l\varphi(x)=\sum_{\xi\neq 0}\mathbb{B}^\dagger(\xi) \widehat{\psi}(\xi) e^{2\pi i\xi\cdot x}\otimes \xi^{\otimes l},
\]
which is a  $0$-homogeneous multiplier of $\psi$, since $\mathbb{B}^\dagger(\cdot)$ is $(-l)$-homogeneous. Hence, by Mikhlin's multiplier theorem, we find that 
\[
\|\nabla^l\varphi\|_{L^p(Q)}=\Big\|\mathcal{F}^{-1}\Big[ \mathbb{B}^\dagger(\frac{\xi}{|\xi|})\widehat{\psi}(\xi) \Big] \Big\|_{L^p}\lesssim \Big\|\mathcal{F}^{-1}\Big[ \widehat{\psi}(\xi) \Big] \Big\|_{L^p}=\|\mathcal{B}\varphi\|_{L^p(Q)}.
\]

For $(iv)$, by working similarly to $(iii)$ we prove that 
\begin{align*}
\|\nabla^{l-1}\varphi\|_{L^p(Q)}\lesssim  \Big\|\mathcal{F}^{-1}\Big[ \frac{1}{|\xi|}\widehat{\psi}(\xi) \Big] \Big\|_{L^p(Q)}=\|\psi\|_{W^{-1,p}(Q)}
\end{align*}
and since $\|\nabla^{l-i}\varphi\|_{L^p}\lesssim \|\nabla^{l-1}\varphi\|_{L^p}$ for $i$=1,..,$l$ we conclude the proof.
\end{proof}

\subsection{$\mathcal{A}$-quasiconvexity}\label{sec:AQ}
Here we recall the definition of $\mathcal{A}$-quasiconvexity and collect results that are used in the sequel. The following definition is due to Fonseca and M\"{u}ller in \cite{Fon}.
\begin{definition}\label{eq: FMdef}
A locally bounded, Borel function $W:\mathbb{R}^N\to\mathbb{R}$ is $\mathcal{A}$-quasiconvex at $\lambda\in\mathbb{R}^N$ if 
\[
\int_Q \big[W(\lambda+\psi(x))-W(\lambda)\big]dx\geq 0,
\]
for all $\psi\in C^\infty(Q)$ such that $\mathcal{A}\psi=0$ and $\int_Q\psi=0$.
\end{definition}
It is proved in \cite{Potentials} that the above definition can equivalently be expressed over arbitrary domains and compactly supported test functions, i.e. it coincides with Dacorogna's definition of $\mathcal{A}$-$\mathcal{B}$ quasiconvexity \cite{Dacorogna} given below. 
\begin{definition}
Let $\Omega\subseteq\mathbb{R}^d$ be a non-empty open subset. A locally bounded, Borel function $W:\mathbb{R}^N\to \mathbb{R}$ is $\mathcal{A}$-quasiconvex at $\lambda\in\mathbb{R}^N$ if and only if
\[
\int_\Omega \big[W(\lambda+\mathcal{B}\varphi(x))-W(\lambda)\big]dx\geq 0,\mbox{ for all }\varphi\in C^\infty_c(\Omega).
\]
\end{definition}

Henceforth, we assume that $W$ has $p$-growth, i.e. $|W(z)|\leq c(1+|z|^p)$.
Then, by density, the above definitions can also be expressed with test functions in $L^p(Q)$ and $W^{l,p}_0(\Omega)$, respectively, where $W^{l,p}_0(\Omega)$ denotes the closure of $C_c^\infty(\Omega)$ in the $W^{l,p}$-norm.

The results presented in this paper, require a strengthened version of the quasiconvexity condition which we now introduce. Let $p\geq 2$ and for $k\in\mathbb{N}$ define the auxiliary function $V:\mathbb{R}^k\to\mathbb{R}$ as
\begin{align}\label{eq:auxf}
    V(z)&:=(|z|^2+|z|^p)^{1/2},\,\,\,z\in\mathbb{R}^k.
\end{align}
If there exists a constant $c_0>0$ such that 
\[
\int_\Omega \big[W(\lambda+\mathcal{B}\varphi(x))-W(\lambda)\big]dx\geq c_0\int_\Omega|V(\mathcal{B}\varphi(x))|^2dx,
\]
for all $\varphi\in W^{l,p}_0(\Omega)$, we say that $W$ is strongly $\mathcal{A}$-quasiconvex at $\lambda\in\mathbb{R}^N$. Equivalently, $W$ is strongly $\mathcal{A}$-quasiconvex at $\lambda\in\mathbb{R}^N$ if
 \[
\int_Q \big[W(\lambda+\psi(x))-W(\lambda)\big]dx\geq c_0\int_Q|V(\psi(x))|^2dx, \label{def:sAqc}
\]
for all $\psi\in L^p(Q)$ with $\mathcal{A}\psi=0$ and $\int_Q\psi=0$. We say that $W$ is (strongly) $\mathcal{A}$-quasiconvex, if it is (strongly) $\mathcal{A}$-quasiconvex at $\lambda$ for all $\lambda\in\mathbb{R}^N$. 

Note that $\mathcal{A}$-quasiconvex functions are not in general continuous as, unlike quasiconvex functions, they are not generally separately convex. However, the condition ${\rm span}\Lambda_{\mathcal{A}}=\mathbb{R}^N$ recovers this loss of separate convexity and then 
 \[
|W(z_1)-W(z_2)|\leq C(1+|z_1|^{p-1}+|z_2|^{p-1})|z_1-z_2|,\mbox{ for all }z_1,z_2\in\mathbb{R}^N.
 \]
The proof can be found in \cite[Lemma 4.4]{NL} and it is based on \cite[Lemma 2.3]{spJan}.

We end this section with a remark on quadratic forms. It is well-known that for these functions rank-one convexity implies quasiconvexity. Similarly, it is not hard to verify that the same holds in the $\mathcal{A}$-quasiconvex setting.

\begin{lemma}\label{lemma:quadratic2}
  Let $M\in\mathbb{R}^{N\times N}$ be a symmetric matrix and define the function $f(\xi)=M\xi\cdot\xi,$ for all $\xi\in\mathbb{R}^N$ . Then if $f$ is convex on the wave cone $\Lambda_\mathcal{A}$, it is also $\mathcal{A}$-quasiconvex.
\end{lemma}
The proof of Lemma \ref{lemma:quadratic2} is almost identical to $\mathcal{A} = {\rm curl}$, see \cite{Dacorogna}.

\section{A G\aa rding-type Inequality for $\mathcal{A}$-quasiconvex Functions}\label{sec:garding}

\subsection{Decomposition Lemma}

The proof of our main result is based on a decomposition lemma which splits a weakly converging sequence into an $\mathcal{A}$-free oscillating and concentrating part. This extends \cite[Theorem 3.4]{JudithR} for the operator $\mathcal{B}$, rather than $\nabla$, and finds its origins in the decomposition results of Kristensen \cite{KrisDec}, and Fonseca and M\"uller \cite{Fon}.
The former of these results is based on the Helmholtz Decomposition, a version of which in the $\mathcal{A}$-free setting can be found in \cite{NL}. Below, we instead use the construction of Fonseca and M\"uller \cite[Lemma 2.14]{Fon} but follow the structure of proof found in \cite{JudithR} to help the reader understand the connection and differences between the ${\rm curl}$-free and $\mathcal{A}$-free cases.

Below we present a crucial result of Fonseca and M\"uller \cite[Lemma 2.14]{Fon} in which the constant rank property is essential and cannot be avoided.

\begin{lemma}\label{FMdec}
Let $\mathcal{A}$ as in \S \ref{sb11}. For every $1<p<+\infty$, there exists a linear and continuous projection operator $\mathcal{P}:L^p(Q)\to L^p(Q)$ and $C>0$ such that
\[
\mathcal{A}(\mathcal{P}v)=0,\quad\int_Q\mathcal{P}v=0\quad\text{and}\quad \|v-\mathcal{P}v\|_{L^p(Q)}\leq C \|\mathcal{A}v\|_{W^{-l,p}(Q)},
\]
for all $v\in L^p(Q)$ with $\int_Q v=0$.
\end{lemma}

To reduce the number of indices in the proof of Lemma \ref{Dec} we assume that the operator $\mathcal{A}$ has order 1 and its potential operator $\mathcal{B}$ has order $l\geq 1$. Nevertheless, the result holds in the general case where the operator $\mathcal{A}$ has order $k\geq 1$ and the proof remains essentially the same.

\begin{lemma} \label{Dec}
Let $2\leq p < +\infty$ and $(\varphi_j)_j\subset W^{l,2}(Q)$ such that $\mathcal{B}\varphi_j\rightharpoonup \mathcal{B}\varphi$ in $L^2(Q)$. Let also $(r_j)_j\subset (0,1)$ such that $(r_j\mathcal{B}\varphi_j)_j$ bounded in $L^p(Q)$. Then, up to a subsequence, there exist sequences $(f_j)_j\subseteq W^{l,2}(Q)$ and $(b_j)_j\subseteq W^{l,2}(Q)$ such that 
\begin{enumerate}
\item[(1)] $\mathcal{B}f_j\rightharpoonup 0$ and $\mathcal{B}b_j\rightharpoonup 0$;
\item[(2)] $(|\mathcal{B}f_j|^2)_j$ is equiintegrable;
\item[(3)] $\mathcal{B}b_j\to 0$ in measure;
\item[(4)] $\mathcal{B}\varphi_{j}=\mathcal{B}\varphi+\mathcal{B}f_j+\mathcal{B}b_j$.
\end{enumerate}
In addition, for a further subsequence, $(f_j)_j$ and $(b_j)_j$ can be chosen so that
\begin{enumerate}
\item[(1\,$^\prime$)] $r_{j}\mathcal{B}f_j\rightharpoonup 0$ and $r_{j}\mathcal{B}b_j\rightharpoonup 0$ in $L^p(Q)$;
\item[(2\,$^\prime$)] $(|r_{j}\mathcal{B}f_j|^p)_j$ is equiintegrable;
\item[(3\,$^\prime$)] $r_{j}\mathcal{B}b_j\to 0$ in measure.
\end{enumerate}
\end{lemma}

\begin{proof}
 By extracting a subsequence, we may assume that $\mathcal{B}\varphi_j\xrightarrow{Y}(\nu_x)_x$ and $r_j\mathcal{B}\varphi_j\xrightarrow{Y}(\mu_x)_x$. The latter notation means that the sequences generate the respective Young measures and in particular that
 \[
 G(\mathcal{B}\varphi_j) \rightharpoonup \langle\nu_x, G\rangle = \int_{\mathbb{R}^N} G(z)\,d\nu_x(z) \mbox{ in }L^1(Q),
 \]
 whenever $(G(\mathcal{B}\varphi_j))$ is equiintegrable, see \cite{Pedregal,Rindler} for details on Young measures. 
 We also observe that, by working with the sequence $\varphi_j-\varphi$ instead of $\varphi_j$, we assume that $\varphi=0$. We split the proof into 4 steps.
\vspace{0.25cm}

\noindent\underline{Step 1. Truncation}: Define, for $k\in\mathbb{N}$, the truncation operator $\tau_k$ by
\[
\tau_k(z):=
     \begin{cases}
        z, &|z|\leq k, \\
        k\,{z}/{|z|}, &|z|>k.
    \end{cases}    
\]
By standard arguments, e.g. \cite[Lemma 2.15]{Fon}, we may find a subsequence such that
\begin{equation}\label{eqq2}
\lim_{k\to\infty}\int_Q|\tau_k(\mathcal{B}\varphi_{j_k})|^2=\int_Q\langle |.|^2,\mathcal{\nu}_x\rangle,
\end{equation}
\begin{equation}\label{eq3}
\lim_{k\to\infty}\int_Q|\tau_k(\mathcal{B}\varphi_{j_k})-\mathcal{B}\varphi_{j_k}|^q=0,
\end{equation}
for $1\leq q<2$. Letting $v_k:=\tau_k(\mathcal{B}\varphi_{j_k})$, it then follows from \eqref{eqq2}, \eqref{eq3} that $(v_k)_k$ is 2-equiintegrable and generates $(\nu_x)_x$. From \eqref{eq3} and the continuity of the operator $\mathcal{A}$, it also follows that $\mathcal{A}v_k\to 0$ in $W^{-1,q}(Q)$.
\vspace{0.25cm}

\noindent\underline{Step 2. Decomposition}: Since $v_k\in L^2(Q)$, we can extend it periodically to $\mathbb{R}^d$ and then apply Lemma \ref{FMdec} to infer that
\[
v_k-\int_Qv_k= F_k+ B_k
\]
where ${F}_k:=\mathcal{P}\Big(v_k-\int_Qv_k\Big)$, ${B}_k:=v_k-\int_Qv_k-\mathcal{P}\Big(v_k-\int_Qv_k\Big)$.
\vspace{0.25cm}

\noindent\underline{Claim 1}: ${B}_k\to 0$ in measure.\vspace{0.2cm}

\noindent By Lemma \ref{FMdec} 
we infer that
\[
\|{B}_k\|_{L^q(Q)}=\|v_k-\int_Qv_k-\mathcal{P}\Big(v_k-\int_Qv_k\Big)\|_{L^q(Q)}
\leq C\|\mathcal{A}v_k\|_{W^{-1,q}(Q)}
\to 0
\]
for all $1\leq q<2$. Hence, ${B}_k\to 0$ in $L^q(Q)$ and so in measure.
\vspace{0.25cm}

\noindent\underline{Claim 2}: $(|{F}_k|^2)_k$ is equiintegrable.\vspace{0.2cm}

\noindent By Step 1, $\Big(v_k-\int_Qv_k\Big)_k$ is 2-equiintegrable, and hence for $\varepsilon>0$ and $q>2$ there exists a sequence $(W_k)_k$ such that 
\[
\|v_k-\int_Qv_k-W_k\|_{L^2(Q)}\leq \varepsilon /C
\]
and $\sup_k\|W_k\|_{L^q(Q)}<+\infty$. This is an equivalent characterisation of equiintegrability, see \cite{KrisDec}. Taking into account the properties of the projection $\mathcal{P}$, we infer that 
\[
\|{F}_k-\mathcal{P}(W_k)\|_{L^2}=\|\mathcal{P}\Big(v_k-\int_Qv_k-W_k\Big)\|_{L^2}\leq C \|v_k-\int_Qv_k-W_k\|_{L^2}\leq \varepsilon 
\]
and 
\[
\sup_k\|\mathcal{P}(W_k)\|_{L^q}\leq C \sup\|W_k\|_{L^q}<+\infty.
\]
This concludes the proof of Claim 2.
\vspace{0.25cm}

\noindent\underline{Claim 3}: ${F}_k,\, {B}_k\rightharpoonup 0$ in $L^2(Q)$.\vspace{0.2cm}

\noindent Since $\mathcal{B}\varphi_{j_k}$ has zero average, \eqref{eq3} and Claim 2 imply that 
\[
{F}_k-\mathcal{B}\varphi_{j_k}=v_k-\int_Qv_k-\mathcal{B}\varphi_{j_k}-{B}_k=v_k-\mathcal{B}\varphi_{j_k}-\int_Q(v_k-\mathcal{B}\varphi_{j_k})-{B}_k\to 0
\]
in measure. In addition, by \eqref{eqq2}, $v_k$ is bounded in $L^2(Q)$ and by the continuity of $\mathcal{P}$, $({F}_k)_k$ is also bounded in $L^2(Q)$ and ${F}_k-\mathcal{B}\varphi_{j_k}\rightharpoonup 0$ in $L^2(Q)$. This proves the claim for $F_k$, since $\mathcal{B}\varphi_{j_k}\rightharpoonup 0$ in $L^2(Q)$. For $({B}_k)_k$ the claim is immediate as it is bounded in $L^2(Q)$ and converges to 0 in measure.
\vspace{0.25cm}

\noindent\underline{Step 3. Concluding the $L^2$-decomposition}: Since ${F}_k$ is $\mathcal{A}$-free with zero average, from Lemma \ref{Ptype} $(i)$, there exists a function $f_k\in W^{l,2}(Q)$ such that ${F}_k=\mathcal{B}f_k$. Set $b_k:=\varphi_{j_k}-f_k$. We thus conclude that
\[
\mathcal{B}b_k=\mathcal{B}\varphi_{j_k}-v_k+\int_Qv_k+{B}_k\to 0
\]
in measure as, by Claim 1, $\mathcal{B}\varphi_{j_k}-v_k\to 0$ in measure. Also, $\int_Qv_k\to 0$ since $\int_Q\mathcal{B}\varphi_{j_k}=0$ and \eqref{eqq2} with $q=1$, and ${B}_k\to 0$ by Claim 1.
Thus, 
\[
\mathcal{B}\varphi_{j_k}=\mathcal{B}f_k+\mathcal{B}b_k
\]
satisfying (1)-(4).
\vspace{0.5cm}

\noindent\underline{Step 4. $L^p$-decomposition}: This follows the arguments in \cite{JudithR} but we include it for completeness. Similarly to Step 1 we can extract a p-equiintegrable subsequence (not relabelled) such that
\begin{equation}\label{eq2}
\lim_{k\to\infty}\int_Q|\tau_k(r_j\mathcal{B}\varphi_j)|^p=\int_Q\langle |.|^p,\mathcal{\mu}_x\rangle,
\end{equation}
and with $v_k=\tau_k(\mathcal{B}\varphi_{j_k})$, we infer that 
\begin{equation*}
|r_{j_k}v_k(x)|=|\tau_{r_{j_k}k}(r_{j_k}\mathcal{B}\varphi_{j_k}(x))|
\leq |\tau_k(r_{j_k}\mathcal{B}\varphi_{j_k}(x))| 
=|\tau_k(r_j\mathcal{B}\varphi_j(x))|,
\end{equation*}
since $r\tau_k(z)=\tau_{rk}(rz)$, $r_{j_k}k\leq k$ and $k\mapsto \tau_k(z)$ is non-decreasing in $z$. Hence, the sequence $(r_{j_k}v_k)_k$ is p-equiintegrable and bounded in $L^p(Q)$. From the linearity and continuity of the projection $\mathcal{P}$, we find that 
\begin{align*}
\mathcal{P}\Big(r_{j_k}v_k-\int_Qr_{j_k}v_k\Big)=r_{j_k}\mathcal{P}\Big(v_k-\int_Qv_k\Big)=r_{j_k}{F}_k
\end{align*} 
and so $\|r_{j_k}{F}_k\|_{L^p(Q)}\lesssim \|r_{j_k}v_k\|_{L^p(Q)}$
which implies that the sequence $(r_{j_k}{F}_k)_k$ is also bounded in $L^p(Q)$. Hence, we can proceed as in Steps 2 and 3 and deduce that $\mathcal{B}f_k$, $\mathcal{B}b_k\rightharpoonup 0$ in $L^p(Q)$. Since $r_{k_j}\in (0,1)$, (3\,$^\prime$) is a straightforward implication of (3).
\end{proof}

\begin{remark}
We remark that the above decomposition applies to any $\mathcal{A}$-free and zero-average sequence $(\psi_j)_j\subseteq L^2(Q)$ with $\psi_j\rightharpoonup \psi$ in $L^2(Q)$. Indeed, by Lemma \ref{Ptype} $(i)$, $\psi_j=\mathcal{B}\varphi_j$, $\psi=\mathcal{B}\varphi$ for some $\varphi_j,\,\varphi\in W^{l,p}(Q)$. In addition, we can choose $b_j$ to be a $\mathbb{B}^\dagger$-primitive and hence to satisfy the bounds of Lemma \ref{Ptype}. Note that $f_j$ is already chosen as a $\mathbb{B}^\dagger$-primitive.

Moreover, we note that the decomposition lemma can also be applied to functions $\varphi_j$ which are defined on an open, bounded domain $\Omega\subset \mathbb{R}^d$. In that case, in Step 3, we need to truncate the functions $f_k$ and so, after the action of the operator $\mathcal{B}$ on the truncated functions, lower order terms will appear. Nevertheless, the strong convergence of the sequence $(f_j)_j$ in $W^{l-1,2}$
is enough to control the lower order terms and conclude the proof.
\end{remark}

\subsection{The G\aa rding-type inequality}

In this section, we prove the G\r{a}rding-type inequality in Theorem \ref{theorem:2}. We assume that $p\geq 2$ and for fixed $K\in\mathbb{R}$, we collect all continuous functions $\bar{U}:\overline{Q}\to\mathbb{R}^N$ in the $K$-ball of $L^{\infty}(Q)$ with a uniform modulus of continuity $\omega$ in the set
\[
\mathcal{U}_K:=\{\bar{U}\in C(\overline{Q}):\|\bar{U}\|_{L^{\infty}(Q)}\leq K,\,|\bar U(x) - \bar U(y)|\leq \omega(|x-y|),\,  \forall\,x,y\in \overline{Q}\}.
\]
Henceforth, we write $C=C(K)$ for any constant uniform for all $\bar{U}\in\mathcal{U}_K$.

Additionally, we assume that $W:\mathbb{R}^N\to\mathbb{R}$ satisfies the following:
\begin{enumerate}
    \item[(H1)] $W\in C^2(\mathbb{R}^N)$;
    \item[(H2)] $W$ is strongly $\mathcal{A}$-quasiconvex;
    \item[(H3)] $|W(z)|\leq c(1+|z|^p)$ and $|DW(z)|\leq c(1+|z|^{p-1})$;
    \item[(H4)] $c(|z|^{p}-1)\leq W(z)$.
\end{enumerate}
\vspace{0.1cm}
\begin{remark}
Recall that, as discussed in \S \ref{sec:AQ}, if $\Lambda_\mathcal{A}$ spans $\mathbb{R}^N$, the growth on $DW$ in (H3) follows from (H1), (H2) and the growth of $W$.
\end{remark}

Next, for $\bar{U}\in\mathcal{U}_K$  we define the function $W(\cdot|\cdot)$ by 
\begin{align*}
W(\bar{U}(x)+z|\bar{U}(x)) & = W(\bar{U}(x)+z) - W(\bar{U}(x)) - DW(\bar{U}(x))\cdot z \\
& = \int_0^1(1-s)D^2W(\bar{U}(x)+sz)\,ds\,z\cdot z.
\end{align*}
We note that this function is related to the relative entropy in the theory of conservation laws and to the Weierstrass excess function in the calculus of variations, see Sections \ref{sec:involutions} and \ref{sec:min}. We also define the auxiliary mapping $\|\cdot\|_{W^{-1,(2,p)}}:L^p(Q)\to\mathbb{R}$ (though not a norm) by
\begin{align}\label{eq:auxpr}
    \|u\|_{W^{-1,(2,p)}}&:=\big(\,\|u\|_{W^{-1,2}(Q)}^2+\|u\|_{W^{-1,p}(Q)}^p\,\big)^{1/2}.
\end{align}

\begin{theorem}\label{theorem:2}
Assume that $W$ satisfies (H1), (H2), (H3) and (H4). There exist constants ${C}_0={C}_0(W,K)>0$, ${C}_1={C}_1(W,K)>0$ such that for all $\bar{U}\in\mathcal{U}_{K}$ and all $\mathcal{A}$-free functions $\psi\in L^{p}(Q)$ with $\int_Q\psi=0$, it holds that
\begin{equation}\label{eq:GardingMain}
\int_Q |V(\psi(x))|^2 dx
\leq C_0 \int_Q W(\bar{U}(x)+\psi(x)|\bar{U}(x)) dx +  C_1\|\psi\|^2_{W^{-1,(2,p)}}.
\end{equation}
\end{theorem}

The main component of the proof Theorem \ref{theorem:2} is presented as Theorem \ref{thm:L^2small} below which is of independent interest in Section \ref{sec:min}.

\begin{theorem}\label{thm:L^2small}
Assume that $W$ satisfies (H1), (H2), (H3) and (H4). There exists $\eps_0>0$ and constants $\tilde{C}_0=\tilde{C}_0(W,K)>0$, $\tilde{C}_1=\tilde{C}_1(W,K)>0$ such that for all $\bar{U}\in\mathcal{U}_K$ and all $\mathcal{A}$-free functions $\psi\in L^{p}(Q)$ with $\int_Q\psi=0$ and  $\|\psi\|_{W^{-1,p}(Q)}<\eps_0$, it holds that
\[
\int_Q |V(\psi(x))|^2dx \leq \tilde{C}_0 \int_QW(\bar{U}(x)+\psi(x)|\bar{U}(x))dx + \tilde{C}_1\|\psi\|^2_{W^{-1,(2,p)}}.
\]

\end{theorem}

We immediately infer Theorem \ref{theorem:2}.

\begin{proof}[Proof of Theorem \ref{theorem:2}]
We claim that for all $\varepsilon>0$ and all $\mathcal{A}$-free and zero-average functions $\psi\in L^p(Q)$ with $\|\psi\|_{W^{-1,p}(Q)}\geq \varepsilon$ it holds that 
\[
\int_Q |V(\psi)|^2 \leq C_0(\varepsilon) \int_Q W(\bar{U}+\psi|\bar{U}) + C_1(\varepsilon)\|\psi\|^2_{W^{-1,(2,p)}},
\]
where $C_0$ and $C_1$ also depend on $\varepsilon$. By Lemma \ref{Ptype} $(i)$ we find $\varphi\in W^{l,p}(Q)$ such that $\psi=\mathcal{B}\varphi$ and by the assumed coercivity of $W$, its smoothness and the fact that $\bar{U}\in\mathcal{U}_K$, we estimate by Young's inequality
\begin{align}
W(\bar{U}+\mathcal{B}\varphi|\bar{U}) &\geq c\left(-1 + |\bar{U}+\mathcal{B}\varphi|^p\right) - C(W,K) - C(\delta)|DW(\bar{U})|^q - \delta |\mathcal{B}\varphi|^p\nonumber \\
& \geq C |\mathcal{B}\varphi|^p - C(W,K),\label{eq:1.1}
\end{align}
for $\delta$ small enough. Note that since $\|\mathcal{B}\varphi\|_{W^{-1,p}(Q)}\geq \eps$, it follows that
\[
C(W,K) \leq \frac{C(W,K)}{\eps^p}\|\mathcal{B}\varphi\|_{W^{-1,p}(Q)}^p
\]
so that, integrating \eqref{eq:1.1} over $Q$ with $|Q|=1$, we infer that
\begin{equation}\label{eq:1.2}
C\int_Q |\mathcal{B}\varphi|^p \leq \int_QW(\bar{U}+\mathcal{B}\varphi|\bar{U}) + \frac{C(W,K)}{\eps^p}\|\mathcal{B}\varphi\|_{W^{-1,p}(Q)}^p.
\end{equation}
However, $\int_Q |V(\mathcal{B}\varphi)|^2 \leq 1 + 2\|\mathcal{B}\varphi\|^p_{L^p}$ and by virtue of the compact embedding $L^p(Q)\hookrightarrow W^{-1,p}(Q)$,
\[
\eps^p \leq  \|\mathcal{B}\varphi\|^p_{W^{-1,p}} \leq C \|\mathcal{B}\varphi\|^p_{L^p},
\]
i.e. $\int_Q |V(\mathcal{B}\varphi)|^2 \leq  C(\eps) \|\mathcal{B}\varphi\|^p_{L^p}$. In particular, \eqref{eq:1.2} says that
\begin{align*}
C(\eps) \int_Q |V(\mathcal{B}\varphi)|^2 
&\leq\int_Q W(\bar{U}+\mathcal{B}\varphi|\bar{U})+\frac{C}{\eps^p}\|\mathcal{B}\varphi\|^2_{W^{-1,(2,p)}},
\end{align*}
which is the desired inequality. Combined with Theorem \ref{thm:L^2small} and choosing $\eps=\eps_0$, we conclude the proof of Theorem \ref{theorem:2}.
\end{proof} 

We next prove a series of results which lead to the proof of Theorem \ref{thm:L^2small}. Lemma \ref{lemma:technical_general} provides some properties of the relative function $W(\cdot|\cdot)$ and its proof can be found in the Appendix. Parts (a)-(c) are collected from \cite{JudithR, JudithKostas,GM09}.

\begin{lemma}\label{lemma:technical_general}
Let $f$ satisfy (H1), (H3), (H4). Then the following hold: 
\begin{itemize}
\item[(a)] There exists $C=C(f,K)$ such that for all $\lambda\in \overline{B(0,K)}$
\[
|f(\lambda+z_1|\lambda) - f(\lambda+z_2|\lambda)| \leq C (|z_1| + |z_2| + |z_1|^{p-1}+|z_2|^{p-1})|z_1 - z_2|.
\]
In particular,
\[
|f(\lambda+z|\lambda)| \leq C|V(z)|^2.
\]
\item[(b)] For every $\delta>0$ there exists $R = R(\delta,f,K)>0$ such that for all $\lambda_1,\,\lambda_2\in \overline{B(0,K)}$ with $|\lambda_1-\lambda_2|<R$, it holds that
\[
|f(\lambda_1+z|\lambda_1) -f(\lambda_2+z|\lambda_2)| \leq \delta |V(z)|^2.
\]
\item[(c)] There exist constants $C=C(f,K)$, $\tilde{C}=\tilde{C}(f,K)$ such that for all $\lambda\in \overline{B(0,K)}$
\[
f(\lambda+z|\lambda) \geq C |z|^p - \tilde{C}|z|^2.
\]
\item[(d)] If $f$ is also strongly convex, i.e. $D^2f(\lambda) z\cdot z\geq \gamma |z|^2$, then there exists $C=C(f,K)$ such that for all $\lambda\in \overline{B(0,K)}$
\[
f(\lambda+z|\lambda) \geq C|V(z)|^2.
\]
\end{itemize}
\end{lemma}

 Next, we define the function $\tilde{W}$ which plays a crucial role in our analysis. It retains the key quasiconvexity property of $W$ in $\overline{B(0,K)}$ and provides the left hand side in the G\r{a}rding inequality \eqref{eq:GardingMain} from Theorem \ref{theorem:2}. 

\begin{lemma}\label{lemma:tWall}
There exists a constant $c_2 = c_2(W,K)$ such that
\[
\tW(z) : = W(z) -  c_2 |V(z)|^2
\]
is $p$-coercive, i.e. $\tW(z) \gtrsim -1 + |z|^p$ and satisfies the following: 
\vspace{0.25cm}

\noindent(1) $\tW$ is strongly $\mathcal{A}$-quasiconvex with constant $c_0/2$ at all $\lambda\in \overline{B(0,K)}$, i.e. for any $Q^\prime\subset Q$ and all $|\lambda|\leq K$, 
\[
\int_{Q^\prime}\tW(\lambda+\nphi) - \tW(\lambda) \geq \frac{c_0}{2} \int_{Q^\prime} |V(\nphi)|^2,\quad\text{for all}\,\,\,\varphi\in W^{l,p}_0(Q^\prime)\,.
\]
\noindent(2) For all $Q^\prime\subset Q$ and $\lambda\in \overline{B(0,K)}$ it holds that
\[
\int_{Q^\prime} D^2\tW(\lambda) \nphi\cdot \nphi \geq c_0 \int_{Q^\prime} |\nphi|^2\quad\text{for all}\,\,\,\varphi\in W^{l,p}_0(Q^\prime).
\]
Equivalently, (1) and (2) can be stated over the torus $Q$ and test functions $\psi\in L^p(Q)$, $\mathcal{A}$-free and zero-average.
\end{lemma}
\begin{proof}
The coercivity of $\tW$ follows from that of $W$ and the fact that $|z|^2\leq 1 + |z|^p$. For (1), let $f(z):= |V(z)|^2$ and note that by Lemma \ref{lemma:technical_general} (a)
\begin{align*}
f(\lambda+\mathcal{B}\varphi) - f(\lambda) = Df(\lambda)\cdot \mathcal{B}\varphi + f(\lambda+\mathcal{B}\varphi|\lambda)
\leq Df(\lambda)\cdot \mathcal{B}\varphi + C|V(\mathcal{B}\varphi)|^2
\end{align*}
for all $|\lambda|\leq K$. Hence, for $\varphi\in W^{l,p}_0(Q^\prime)$, noting that $\int_{Q^\prime}\mathcal{B}\varphi=0$, 
\begin{align*}
\int_{Q^\prime} |V(\lambda + \mathcal{B}\varphi)|^2 - |V(\lambda)|^2 & \leq C\int_{Q^\prime} |V(\psi)|^2.
\end{align*}
Using again that $\int_{Q^\prime}\mathcal{B}\varphi=0$, by the strong $\mathcal{A}$-quasiconvexity of $W$,
\begin{align*}
\int_{Q^\prime} \tW(\lambda + \mathcal{B}\varphi) - \tW(\lambda)& \geq  c_0 \int_{Q^\prime}|V(\mathcal{B}\varphi)|^2 - c_2 C\int_{Q^\prime}|V(\mathcal{B}\varphi)|^2.
\end{align*}
Hence, choosing $c_2 \leq c_0/(2C)$, we conclude the proof of (1). 

For (2), fix $\lambda\in\Rdd$, $|\lambda|\leq K$, and note that $\mathcal{A}$-quasiconvexity says that $I(0) \leq I(\mathcal{B}\varphi)$ for all $\varphi\in W^{l,p}_0(Q^\prime)$,
where
\[
I(z) := \int_{Q^\prime} \tW(\lambda + z) - \tW(\lambda) - \frac{c_0}{2}  |V(z)|^2.
\]
Hence, for all $\varphi\in W^{l,p}_0(Q^\prime)$,
\[
0 \leq \frac{d^2}{d\eps^2} I(\eps\mathcal{B}\varphi)|_{\eps = 0} = \int_{Q^\prime} D^2 \tW(\lambda)\nphi\cdot \nphi - c_0|\nphi|^2.
\]
This concludes the proof of the lemma.
\end{proof}

In the next proposition we prove a G\aa rding-type inequality based on the $\Lambda_\mathcal{A}$-convexity of $\tilde{W}$ which is crucial for the contradiction argument of the proof of Theorem \ref{thm:L^2small}. The proof follows the arguments of \cite[Lemma 4.3]{Daf} and \cite{Giaquinta}. Note that compared to \cite{Daf}, since $\bar{U} \in \mathcal{U}_K$, we do not need to assume any smallness on the local oscillations.

\noindent\begin{proposition}\label{prop:1}
For every $\delta>0$, there exists a constant $c_1 = c_1(W,K,\delta)$ such that for all $\bar{U}\in\mathcal{U}_K$ and $\varphi\in W^{l,p}(Q)$
\[
\int_{Q} D^2\tW(\bar{U}(x))\nphi\cdot \nphi \geq c_0(1-\delta) \int_Q |\nphi|^2 - c_1 \sum_{i=1}^l\int_Q|\nabla^{l-i}\varphi|^2.
\]
\end{proposition}

\begin{proof}
Fix $\delta>0$ and pick a finite cover $\{Q_i\}\subset Q$, $Q_i=Q_i(x_i,r_i)$, such that
\[
|D^2\tW(\bar{U}(x)) - D^2\tW(\bar{U}(x_i))| \leq c_0 \delta(1-\delta)^2.
\]
Note that since $\bar{U}\in\mathcal{U}_K$ are bounded with a uniform modulus of continuity, and $\tW\in C^2(\Rdd)$ the cover can be chosen uniformly for any $\bar{U}\in\mathcal{U}_K$. 

Next, choose a partition of unity $\{\rho_i\}$ subordinate to the cover $\{Q_i\}$ such that $\rho_i\in C^\infty_c(Q_i)$ and $\sum_i \rho_i^2 = 1$. Given $\varphi\in W^{l,p}(Q)$, 
\begin{align*}
\int_{Q} D^2\tW(\bar{U}(x))\nphi\cdot \nphi & = \sum_i \int_{Q_i}\rho_i^2 D^2\tW(\bar{U}(x_i))\nphi\cdot \nphi\nonumber\\
& + \sum_i \int_{Q_i}\rho_i^2 \left[D^2\tW(\bar{U}(x)) - D^2\tW(\bar{U}(x_i))\right]\nphi\cdot \nphi 
\end{align*}
so that, by the choice of the cover, and for all $\bar{U}\in\mathcal{U}_K$,
\begin{align}
\int_{Q} D^2\tW(\bar{U}(x))\nphi\cdot \nphi & \geq \sum_i \int_{Q_i} D^2\tW(\bar{U}(x_i))(\rho_i\nphi)\cdot(\rho_i\nphi)\nonumber\\
& \quad - c_0\delta(1-\delta)^2 \int_{Q}|\nphi|^2. 
\label{eq:1}
\end{align}
Note that $\rho_i\nphi = \mathcal{B}(\rho_i\varphi) - \sum_{j=1}^lB^L_j[\nabla^j\rho_i,\nabla^{l-j}\varphi]$, where $B^L_j$ are given by the Leibniz rule. However, $\rho_i\varphi\in W^{l,p}_0(Q_i)$ and $|\bar{U}(x_i)|\leq K$ so that by Lemma \ref{lemma:tWall}
\begin{equation} \label{eq:2}
\int_{Q_i} D^2\tW(\bar{U}(x_i))\mathcal{B}(\rho_i\varphi)\cdot\mathcal{B}(\rho_i\varphi) \geq c_0 \int_{Q_i} |\mathcal{B}(\rho_i\varphi)|^2.
\end{equation}
Moreover, note that
\begin{equation}\label{eq:2b}
\|\sum_{j=1}^lB^L_j[\nabla^j\rho_i,\nabla^{l-j}\varphi]\|_{L^2(Q_i)}^2 \leq C(\sup_j\|\nabla^j\rho_i\|_\infty)\sum_{j=1}^l\int_{Q_i}|\nabla^{l-j}\varphi|^2 .
\end{equation}
Then,
\begin{align*}
&\int_{Q_i}\rho_i^2 D^2\tW(\bar{U}(x_i))\nphi\cdot\nphi = \int_{Q_i} D^2\tW(\bar{U}(x_i))\mathcal{B}(\rho_i\varphi)\cdot\mathcal{B}(\rho_i\varphi) \\
 &+ \int_{Q_i} D^2\tW(\bar{U}(x_i))\Big(\sum_{j=1}^lB^L_j[\nabla^j\rho_i,\nabla^{l-j}\varphi]\Big)\cdot\Big(\sum_{j=1}^lB^L_j[\nabla^j\rho_i,\nabla^{l-j}\varphi]\Big)\\&
 - 2 \int_{Q_i} D^2\tW(\bar{U}(x_i))\mathcal{B}(\rho_i\varphi)\cdot\Big(\sum_{j=1}^lB^L_j[\nabla^j\rho_i,\nabla^{l-j}\varphi]\Big)
 =: I + II + III.
\end{align*}
By \eqref{eq:2} and \eqref{eq:2b}, we find that
\[
I \geq c_0 \int_{Q_i}|\mathcal{B}(\rho_i\varphi)|^2,\quad II \geq -C\sum_{j=1}^l\int_{Q_i}|\nabla^{l-j}\varphi|^2
\]
where $C = C(\tW,K)$. For term $III$, Young's inequality and \eqref{eq:2b} say that
\begin{equation*}
-III  \leq c_0\delta \int_{Q_i}|\mathcal{B}(\rho_i\varphi)|^2 + C\sum_{j=1}^l\int_{Q_i}|\nabla^{l-j}\varphi|^2,
\end{equation*}
where $C = C(\tW,K,\delta)$. Putting these together we deduce that
\begin{equation}\label{eq:3}
\int_{Q_i}\rho_i^2 D^2\tW(\bar{U}(x_i))\nphi\cdot\nphi \geq c_0(1-\delta)\int_{Q_i} |\mathcal{B}(\rho_i\varphi)|^2 - C\sum_{j=1}^l\int_{Q_i}|\nabla^{l-j}\varphi|^2,
\end{equation}
for all $\bar{U}\in\mathcal{U}_K$. Applying Young's inequality again,
\begin{equation*}
\int_{Q_{i}}|\mathcal{B}(\rho_i\varphi)|^2\geq (1-\delta)\int_{Q_{i}}\rho_i^2|\nphi|^2-\frac{C}{\delta}\sum_{j=1}^l\int_{Q_{i}}|\nabla^{l-j}\varphi|^2,
\end{equation*}
where $C$ only depends on the cover. Now \eqref{eq:3} reads as,
\[
\int_{Q_i}\rho_i^2 D^2\tW(\bar{U}(x_i))\nphi\cdot\nphi \geq c_0(1-\delta)^2\int_{Q_i} \rho_i^2|\mathcal{B}\varphi|^2 - C(\delta)\sum_{j=1}^l\int_{Q_{i}}|\nabla^{l-j}\varphi|^2.
\]
After summing up, \eqref{eq:1} results in
\begin{equation*}
\int_{Q} D^2\tW(\bar{U}(x))\nphi\cdot\nphi \geq c_0(1-\delta)^3\int_{Q} |\mathcal{B}\varphi|^2 - C(\delta)\sum_{j=1}^l\int_Q|\nabla^{l-j}\varphi|^2.
\end{equation*}
This concludes the proof.
\end{proof}

We next prove a central proposition which is an equivalent characterisation of $\mathcal{A}$-quasiconvexity in $\overline{B(0,K)}$. It can be seen as a limiting version of a G\aa rding inequality which replaces the $\mathcal{A}$-quasiconvexity condition in the proof of Theorem \ref{thm:L^2small}. Its proof follows \cite{JudithR, JudithKostas} and relies on an observation in \cite{Zhang} that smooth extremals are spatially local minimisers.

\begin{proposition}\label{prop:2}
Let $\left(\bar{U}_k\right)_k\subset \mathcal{U}_K$, $(h_k)_k\subset W^{l,p}(Q)$, $(a_k)_k\subset \R$ such that
\begin{itemize}
\item $a_k^{-1}V(\nabla^{l-i}h_k) \rightarrow 0$ strongly in $L^2(Q)$ for all i=1,..,$l$,
\item $\left(a_k^{-1}V(\mathcal{B} h_k)\right)_k$ is bounded in $L^2(Q)$.
\end{itemize}
Then,
\[
\liminf_{k\to\infty} \frac{c_0}{4} a_k^{-2}\int_Q |V(\mathcal{B} h_k(x))|^2dx \leq \liminf_{k\to\infty} a_k^{-2}\int_Q \tW(\bar{U}_k(x)+\mathcal{B}h_k(x)|\bar{U}_k(x))dx.
\]
\end{proposition}

\begin{proof}


Observe that by Lemma \ref{lemma:technical_general} (b), letting $\delta = c_0/4$ we find $R=R(c_0,\tW,K)$ such that for all $\bar{U}\in\mathcal{U}_K$ and whenever $|x-x_0|<R$
\[
|\tW(\bar{U}(x)+z|\bar{U}(x)) - \tW(\bar{U}(x_0)+z|\bar{U}(x_0))|\leq \frac{c_0}{4} |V(z)|^2.
\]
Indeed, since $\bar U\in\mathcal{U}_K$, this follows by the assumed growth on $W$, (H3). In particular,
for $x\in Q(x_0,R)$, let $z=\mathcal{B}\varphi(x)$ where $\varphi\in W^{l,p}_0(Q(x_0,R))$ and integrate to find that
\begin{align}
\int_{Q(x_0,R)} \tW(\bar{U}(x)+\mathcal{B}\varphi|\bar{U}(x)) & \geq \int_{Q(x_0,R)} \tW(\bar{U}(x_0)+\mathcal{B}\varphi|\bar{U}(x_0)) - \frac{c_0}{4}|V(\mathcal{B}\varphi)|^2 \nonumber \\
& \geq \frac{c_0}{4}\int_{Q(x_0,R)}|V(\mathcal{B}\varphi)|^2,\label{lemma:Zhang}
\end{align}
by the strong $\mathcal{A}$-quasiconvexity of $\tW$ in $\overline{B(0,K)}$, see Lemma \ref{lemma:tWall}, and the fact that $\int_{Q(x_0,R)} D\tW(\bar{U}(x_0))\cdot\mathcal{B}\varphi = 0$.
Next, note that since $\left(a_k^{-1}V(\mathcal{B} h_k)\right)_k$ is bounded in $L^2(Q)$ we may assume that (up to a subsequence)
\[
a_k^{-2} |V(\mathcal{B} h_k)|^2 \mathcal{L}^d\mres Q \overset{\ast}\rightharpoonup \mu,\quad\mbox{in }\mathcal{M}(\overline{Q}) = \left(C(\overline{Q})\right)^\ast.
\]
Since $\mu$ is a positive measure, there can be at most a countable number of hyperplanes parallel to the coordinate axes which admit non-null $\mu$-measure. Hence, we can extract a finite cover of $Q$ by cubes $Q(x_j,r_j)$ with the property that $r_j<R$, so that \eqref{lemma:Zhang} applies and that
\begin{equation}\label{eq:muzero}
\mu(\overline{Q}\cap \partial Q(x_j,r_j)) = 0.
\end{equation}
Next, consider cut-off functions $\rho_j \in C^\infty_c(Q(x_j,r_j))$ such that for $\lambda\in(0,1)$
\[
\mathbbm{1}_{Q(x_j,\lambda r_j)} \leq \rho_j \leq \mathbbm{1}_{Q(x_j, r_j)},\,\, \|\nabla^i\rho_j\|_{L^{\infty}(Q)}\leq \frac{C}{(1-\lambda)^i},
\]
for i=1,..,$l$. Let $\varphi = \rho_j h_k\in W^{l,p}_0(Q(x_j,r_j))$ in \eqref{lemma:Zhang} to find that
\[
\frac{c_0}{4}\int_{Q(x_j,r_j)}|V(\mathcal{B}(\rho_j h_k))|^2 \leq \int_{Q(x_j,r_j)}\tW(\bar{U}_k+\mathcal{B}(\rho_jh_k)|\bar{U}_k),
\]
where $\bar{U}_k\in\mathcal{U}_K$. Thus, by Lemma \ref{lemma:technical_general} (a) and for $C=C(\tW,K)$,
\begin{multline*}
\frac{c_0}{4}\int_{Q(x_j,\lambda r_j)}|V(\mathcal{B}h_k)|^2 + \frac{c_0}{4}\int_{Q(x_j,r_j)\setminus Q(x_j,\lambda r_j)}|V(\mathcal{B}(\rho_j h_k))|^2 \\
\leq \int_{Q(x_j,\lambda r_j)}\tW(\bar{U}_k+\mathcal{B}h_k|\bar{U}_k)+ \int_{Q(x_j,r_j)\setminus Q(x_j,\lambda r_j)}\tW(\bar{U}_k+\mathcal{B}(\rho_jh_k)|\bar{U}_k) \\
\leq \int_{Q(x_j,\lambda r_j)} \tW(\bar{U}_k+\mathcal{B}h_k|\bar{U}_k)+ C\int_{Q(x_j,r_j)\setminus Q(x_j,\lambda r_j)}|V(\mathcal{B}(\rho_j h_k))|^2.
\end{multline*}
Since the second term on the left-hand side is positive, we infer that
\begin{multline*}
\frac{c_0}{4}\int_{Q(x_j,\lambda r_j)}|V(\mathcal{B}h_k)|^2 \leq \int_{Q(x_j,\lambda r_j)} \tW(\bar{U}_k+\mathcal{B}h_k|\bar{U}_k) \\
+ C\int_{Q(x_j,r_j)\setminus Q(x_j,\lambda r_j)}|V(\mathcal{B} h_k)|^2 + \sum_{i=1}^l\left| V\left(\frac{\nabla^{l-i} h_k}{(1-\lambda)^i}\right)\right|^2,
\end{multline*}
as $\rho_j\in[0,1]$. Summing over $j$, we deduce that
\begin{multline*}
\frac{c_0}{4}\int_{Q}|V(\mathcal{B}h_k)|^2 - \frac{c_0}{4}\sum_j \int_{Q(x_j,r_j)\setminus Q(x_j,\lambda r_j)}|V(\mathcal{B} h_k)|^2\\
\leq \int_{Q} \tW(\bar{U}_k+\mathcal{B}h_k|\bar{U}_k) - \sum_j \int_{Q(x_j,r_j)\setminus Q(x_j,\lambda r_j)} \tW(\bar{U}_k+\mathcal{B}h_k|\bar{U}_k)\\
+ C\sum_j \int_{Q(x_j,r_j)\setminus Q(x_j,\lambda r_j)} |V(\mathcal{B} h_k)|^2 + \sum_{i=1}^l\left| V\left(\frac{\nabla^{l-i} h_k}{(1-\lambda)^i}\right)\right|^2,
\end{multline*}
so that by Lemma \ref{lemma:technical_general} (a),
\begin{multline*}
\frac{c_0}{4}\int_{Q}|V(\mathcal{B}h_k)|^2 \leq \int_{Q} \tW(\bar{U}_k+\mathcal{B}h_k|\bar{U}_k)\\
+ C\sum_j \int_{Q(x_j,r_j)\setminus Q(x_j,\lambda r_j)} |V(\mathcal{B} h_k)|^2 + \sum_{i=1}^l\left| V\left(\frac{\nabla^{l-i} h_k}{(1-\lambda)^i}\right)\right|^2.
\end{multline*}
Next, multiply by $a_k^{-2}$ and take the limit $k\to\infty$ to obtain
\begin{multline*}
\liminf_{k\to\infty} \frac{c_0}4 a_k^{-2} \int_Q |V(\mathcal{B} h_k)|^2 \leq \liminf_{k\to\infty} a_k^{-2} \int_Q\tW(\bar{U}_k+\mathcal{B}h_k|\bar{U}_k)\\
+ C\limsup_{k\to\infty} \sum_j \int_{Q(x_j,r_j)\setminus Q(x_j,\lambda r_j)} a_k^{-2}|V(\mathcal{B} h_k)|^2 + a_k^{-2}\sum_{i=1}^l\left| V\left(\frac{\nabla^{l-i}h_k}{(1-\lambda)^i}\right)\right|^2\\
\leq \liminf_{k\to\infty} a_k^{-2} \int_Q\tW(\bar{U}_k+\mathcal{B}h_k|\bar{U}_k)\\
+ C\sum_j \mu\left(\overline{Q}\cap\left(\overline{Q(x_j,r_j)}\setminus Q(x_j,\lambda r_j)\right)\right),
\end{multline*}
since $a_k^{-1}V(\nabla^{l-i}h_k)\rightarrow 0$ in $L^2(Q)$ and $a_k^{-2} |V(\mathcal{B} h_k)|^2 \mathcal{L}^d\mres Q \overset{\ast}\rightharpoonup \mu$ in $\mathcal{M}(Q)$. Take the limit $\lambda\to 1$ to complete the proof after noting \eqref{eq:muzero}.
\end{proof}

We may now prove Theorem \ref{thm:L^2small}. Note that all primitive functions constructed in the proof are $\mathbb{B}^\dagger$-primitives and satisfy the bounds of Lemma \ref{Ptype}. Otherwise, the loss of control of the full Sobolev norm, prevents the application of Proposition \ref{prop:2}.

\begin{proof}[Proof of Theorem \ref{thm:L^2small}]
We show the following: there exists $\eps_0>0$ such that for all $\psi\in L^{p}(Q)$, $\mathcal{A}$-free and zero-average, and $\bar{U}\in\mathcal{U}_K$, $\|\psi\|_{W^{-1,p}(Q)}<\eps_0$ implies that
\begin{equation}\label{eq:tGF>0}
\int_Q\tW(\bar{U}+\mathcal{B}\varphi|\bar{U}) + \frac{c_1}{2}\sum_{i=1}^l\int_Q|\nabla^{l-i}\varphi|^2 \geq 0,
\end{equation}
where $\varphi$ is the $\mathbb{B}^\dagger$-primitive of $\psi$ whose existence is guaranteed by Lemma \ref{Ptype}. Then, since $f(z)=|V(z)|^2$ is strongly convex, Lemma \ref{lemma:technical_general} (d) says that for $C=C(p,K)$
\begin{align*}
C\int_Q |V(\nphi)|^2 & \leq
c_2\int_Q  f(\bar{U}+\mathcal{B}\varphi|\bar{U}) \\
& \leq \int_Q W(\bar{U}+\mathcal{B}\varphi|\bar{U}) + \frac{c_1}{2}\sum_{i=1}^l\int_Q|V(\nabla^{l-i}\varphi)|^2.
\end{align*}
This concludes the proof of Theorem \ref{thm:L^2small} since by Lemma \ref{Ptype} $(iv)$,
\[
\sum_{i=1}^l\int_Q|V(\nabla^{l-i}\varphi)|^2=\|\varphi\|^2_{W^{l-1,2}(Q)}+\|\varphi\|^p_{W^{l-1,p}(Q)}\leq C\|\mathcal{B}\varphi\|^2_{W^{-1,(2,p)}}.
\]

We proceed to prove \eqref{eq:tGF>0} by contradiction. Suppose \eqref{eq:tGF>0} is false. Then, there exist $(\bar{U}_k)_k\subset \mathcal{U}_K$ and pairs $(\psi_k,\varphi_k)\subseteq L^p(Q)\times W^{l,p}(Q)$  with
\[
\|\psi_k\|_{W^{-1,p}(Q)} \rightarrow 0,\, \bar{U}_k \overset{\ast}{\rightharpoonup} \bar{U} \mbox{ in }L^{\infty}(Q)
\]
such that
\begin{equation}\label{eq:6}
\int_Q \tW(\bar{U}_k+\mathcal{B}\varphi_k|\bar{U}_k) + \frac{c_1}2\sum_{i=1}^l\int_Q|\nabla^{l-i}\varphi_k|^2 <0,
\end{equation}
where $\varphi_k$ is the $\mathbb{B}^\dagger$-primitive of $\psi_k$. Note that $\|\varphi_k\|_{W^{l-1,p}}\lesssim \|\psi_k\|_{W^{-1,p}} \rightarrow 0$ by Lemma \ref{Ptype} $(iv)$, and $\bar{U}_k\rightarrow \bar{U}$ in $C^0(Q)$ by the Arzel\`a-Ascoli theorem with $\bar{U}\in\mathcal{U}_K$.
We split the proof into 5 steps.\\\quad\\
\underline{Step 1}: Let $\alpha_k = \|\nphi_k\|_{L^2(Q)}$, $\beta_k = \|\nphi_k\|_{L^p(Q)}$. We show that $\alpha_k,\,\beta_k\to 0$, as $k\to\infty$ and
\begin{equation}\label{eq:Lambdabounded}
\sup_k \frac{\beta_k^p}{\alpha_k^2} =:\Lambda <\infty.
\end{equation}

To show that $\alpha_k,\,\beta_k\to 0$, recall that, by Lemma \ref{lemma:tWall} (a), $\tW$ is $p$-coercive and, as in the proof of Theorem \ref{theorem:2} with $\tW$ instead of $W$, we may estimate
\begin{align*}
\tW(\bar{U}_k+\mathcal{B}\varphi_k|\bar{U}_k) \geq  - C(\tW,K) + c  |\mathcal{B}\varphi_k|^p,
\end{align*}
which, combined with \eqref{eq:6}, states that
$(\mathcal{B}\varphi_k)_k$ is bounded in $L^{p}(Q)$.
By Proposition \ref{prop:2} with $a_k=1$ and $h_k = \varphi_k$, since
\[
a_k^{-1}V(\nabla^{l-i}h_k) = V(\nabla^{l-i}\varphi_k) \rightarrow 0\mbox{ in }L^2(Q),\,\,\forall\,i=1,..,l
\]
and $(a_k^{-1}V(\mathcal{B} h_k))_k = (V(\mathcal{B}\varphi_k))_k$ is bounded in $L^2(Q)$, we find that
\begin{align*}
\liminf_{k\to\infty} \frac{c_0}4\int_Q|V(\nphi_k)|^2  \leq \liminf_{k\to\infty}\int_Q\tW(\bar{U}_k+\mathcal{B}\varphi_k|\bar{U}_k)\leq 0
\end{align*}
by \eqref{eq:6}. In particular, up to a subsequence, $\alpha_k$, $\beta_k\to 0$.
Regarding the bound on $\beta_k^p/\alpha_k^2$, Lemma \ref{lemma:technical_general} (c) and the coercivity of $\tW$ imply that
\begin{equation}\label{eq:new_coercivity}
\tW(\bar{U}_k+\mathcal{B}\varphi_k|\bar{U}_k) \geq d |\nphi_k|^p - c |\nphi_k|^2,
\end{equation}
for constants $d,\,c>0$ uniform for $\bar{U}\in\mathcal{U}_K$. Dividing by $\alpha_k^2$, we infer that
\[
d \frac{\beta_k^p}{\alpha_k^2} - c \leq \alpha_k^{-2} \int_Q  \tW(\bar{U}_k+\mathcal{B}\varphi_k|\bar{U}_k) < 0,
\]
by \eqref{eq:6} which concludes Step 1.
Note that \eqref{eq:6} implies that $\alpha_k\neq 0$.\\\quad\\
\underline{Step 2}: Following \cite{JudithR, JudithKostas, GM09}, we decompose the normalised sequence
\[
w_k := \alpha_k^{-1} \varphi_k.
\]
Moreover, $\|\mathcal{B}w_k\|_{L^2(Q)} = 1$,  $\int_Q\mathcal{B}w_k=0$,  $\mathcal{A}(\mathcal{B}w_k) = 0$ and we can find $w\in W^{l,p}(Q)$ such that $\mathcal{B}w_k\rightharpoonup \mathcal{B}w$ in $L^{2}(Q)$. Setting 
\[
\eta_k = \frac{\alpha_k}{\beta_k}\in(0,1],
\]
we also infer that $(\eta_k\mathcal{B}w_k)_k$ is bounded in $L^{p}(Q)$ with $\|\eta_k\mathcal{B}w_k\|_{L^p} = 1$. We may thus apply Lemma \ref{Dec} to find $\mathbb{B}^\dagger$-primitives $f_k,\,b_k \in W^{l,2}(Q)$ such that
\begin{enumerate}
\item[(a)] $\mathcal{B}w_k = \mathcal{B}w + \mathcal{B}f_k + \mathcal{B}b_k$;
\item[(b)] $\mathcal{B}f_k\rightharpoonup 0$,  $\mathcal{B} b_k \rightharpoonup 0$ in $L^{2}(Q)$, and $\eta_k \mathcal{B}f_k\rightharpoonup 0$, $\eta_k\mathcal{B} b_k \rightharpoonup 0$ in $L^{p}(Q)$;
\item[(c)] $\mathcal{B} b_k \rightarrow 0$ and $\eta_k \mathcal{B} b_k \rightarrow 0$ in measure;
\item[(d)] $\left(|\mathcal{B} f_k|^2\right)_k$ and $\left(|\eta_k \mathcal{B} f_k|^p\right)_k$ are equiintegrable.
\end{enumerate}
Write
\begin{equation}\label{eq:7}
g_k(x) = \alpha_k^{-2}\left[\tW(\bar{U}_k+\alpha_k\mathcal{B}w_k|\bar{U}_k)-\tW(\bar{U}_k+\alpha_k\mathcal{B} b_k|\bar{U}_k)\right]
\end{equation}
and note that, since $\alpha_kw_k = \varphi_k$,
\begin{equation*}\label{eq:part}
\begin{split}
&\int_Q g_k(x) + \alpha_k^{-2} \tW(\bar{U}_k+\alpha_k\mathcal{B} b_k|\bar{U}_k) + \frac{c_1}2\sum_{i=1}^l\int_Q|\nabla^{l-i}w_k|^2 \\ 
=& \alpha_k^{-2}\left(\int_Q \tW(\bar{U}_k+\mathcal{B} \varphi_k|\bar{U}_k) + \frac{c_1}2\sum_{i=1}^l\int_Q|\nabla^{l-i}\varphi_k|^2  \right)< 0. 
\end{split}
\end{equation*}
The idea in the proof of \cite{JudithR} is to show that quasiconvexity forces the contribution of the concentrating part $\alpha_k^{-2} \int_Q \tW(\bar{U}_k+\alpha_k\mathcal{B} b_k|\bar{U}_k)$ to be nonnegative and thus the contribution of the oscillating part $\int_Q g_k$ must be negative by \eqref{eq:part}. Step 4, shows that the latter bounds a Young measure version of the second variation which is hence itself negative. This contradicts Proposition \ref{prop:1} in Step 5, noting that this is the only point where Proposition \ref{prop:1} is used.\\\quad\\
\underline{Step 3}: In this step we show that the contribution of the concentrating part must be nonnegative in the limit due to $\mathcal{A}\text{-quasiconvexity}$. In particular, we prove that
\begin{equation}\label{eq:7b}
\liminf_{k\to\infty}\alpha_k^{-2} \int_Q \tW(\bar{U}_k+\alpha_k\mathcal{B} b_k|\bar{U}_k) \geq 0,
\end{equation}
as a consequence of Proposition \ref{prop:2}. Combined with \eqref{eq:part} and the fact that  $\nabla^{l-i}w_k\rightarrow \nabla^{l-i}w$ for all i=1,..,$l$ strongly in $L^2(Q)$, this says that
\begin{equation}\label{eq:8}
\frac{c_1}2\sum_{i=1}^l\int_Q|\nabla^{l-i}w|^2 + \liminf_{k\to\infty} \int_Q g_k(x) \leq 0.
\end{equation}
To prove \eqref{eq:7b}, simply apply Proposition \ref{prop:2} with $a_k = \alpha_k$ and $h_k = \alpha_k b_k$ after noting that
\[
\alpha_k^{p-2} = \frac{\beta_k^p}{\alpha_k^2}\eta_k^p = \Lambda \eta_k^p,
\]
where, by Step 1, $\Lambda = \beta_k^p/\alpha_k^2$ is bounded. Thus, again due to the control of the full Sobolev norm of the $\mathbb{B}^\dagger$-primitives $b_k$, Lemma \ref{Ptype}, we infer that
\[
a_k^{-2}|V(\alpha_k\nabla^{l-i}b_k)|^2 = |\nabla^{l-i}b_k|^2 + \Lambda |\eta_k \nabla^{l-i}b_k|^p \rightarrow 0\mbox{ in }L^1(Q),
\]
for i=1,..,$l$. Also, $\Big(a_k^{-2}|V(\mathcal{B} h_k)|^2\Big)_k =\Big( |\mathcal{B} b_k|^2 + \Lambda |\eta_k \mathcal{B} b_k|^p\Big)_k$  is bounded in $L^1(Q)$. So, Proposition \ref{prop:2} says that
\begin{align*}
0\leq \liminf_{k\to\infty} \frac{c_0}4\int_Q\alpha_k^{-2} |V(\alpha_k \mathcal{B} b_k)|^2 \leq \liminf_{k\to\infty} \alpha_k^{-2}\int_Q\tW(\bar{U}_k+\alpha_k\mathcal{B} b_k|\bar{U}_k).
\end{align*}
\underline{Step 4}: Next, consider the $\mathcal{A}$-p-Young measure generated by the sequence $\mathcal{B}w_k$, say $\nu = (\nu_x)_{x\in Q}$, and recall that $\bar{U}_k \rightarrow \bar{U}$ in $C^0(Q)$. In this Step we show that
\begin{equation}\label{eq:8a}
\frac12 \int_Q \langle \nu_x, D^2\tW(\bar{U}(x))z\cdot z\rangle \leq \liminf_{k\to\infty} \int_Q g_k(x).
\end{equation}
In particular, in conjunction with \eqref{eq:8}, we infer that
\begin{equation}\label{eq:9}
\frac{1}2\sum_{i=1}^l\int_Q c_1 |\nabla^{l-i}w|^2 +  \langle \nu_x, D^2\tW(\bar{U}(x))z\cdot z\rangle \leq 0.
\end{equation}
In Step 5 we show how \eqref{eq:9} leads to a contradiction. 

To show \eqref{eq:8a} we first prove the equiintegrability of $(g_k)_k$ in \eqref{eq:7}. By Lemma \ref{lemma:technical_general} (a) and a constant $C=C(\tW,K)$, Young's inequality gives
\begin{align*}
|g_k| & \leq C(|\mathcal{B}w_k|+ |\mathcal{B} b_k| + \alpha_k^{p-2}|\mathcal{B}w_k|^{p-1} + \alpha_k^{p-2}|\mathcal{B} b_k|^{p-1})|\mathcal{B}w_k - \mathcal{B} b_k|\\
& \leq C\delta (|\mathcal{B}w_k|^2+ |\mathcal{B} b_k|^2) + C(\delta)| \mathcal{B}(w+f_k)|^2\\
&\quad + C\delta(\alpha_k^{p-2}|\mathcal{B}w_k|^p+ \alpha_k^{p-2}|\mathcal{B} b_k|^p) + C(\delta) \alpha_k^{p-2} | \mathcal{B}(w+f_k)|^p,
\end{align*}
recalling that, by Lemma \ref{Dec}, $\mathcal{B}w_k - \mathcal{B} b_k = \mathcal{B}(w+f_k)$. However, by the same lemma, $(\mathcal{B}w_k)_k$ and $(\mathcal{B}b_k)_k$ are bounded in $L^{2}(Q)$, \text{and} $(|\mathcal{B}(w+f_k)|^2)_k$ is equiintegrable. Similarly, $\alpha_k^{p-2} = \Lambda \eta_k^p$, where, by Step 1, $\Lambda = \beta_k^p/\alpha_k^2$ is bounded. Thus
$(\alpha_k^{p-2}|\mathcal{B}w_k|^p)_k$ and $(\alpha_k^{p-2}|\mathcal{B} b_k|^p)_k$ are bounded in $L^1(Q)$ and $(\alpha_k^{p-2} | \mathcal{B}(w+f_k)|^p)_k$ is equiintegrable. Hence,
given a set $A\subset Q$
\[
|g_k|\leq \delta C + C(\delta)\int_A|\mathcal{B}(w+f_k)|^2 + C(\delta)\int_A\alpha_k^{p-2} | \mathcal{B}(w+f_k)|^p
\]
and so $(g_k)_k$ is also equiintegrable. Then, for $\eps > 0$ fixed, we can find $m_\eps$ such that
\[
\int_{\{|\mathcal{B}w_k|\geq m\}\cup\{|\mathcal{B} b_k|\geq m\}} |g_k| < \eps,\mbox{ for all } m\geq m_\eps.
\]
This indeed follows from the fact that $\mathcal{B} b_k\rightarrow 0$ in measure and that
\[
\lim_{R\to\infty} \sup_k \Big|\{x\in Q : |\mathcal{B}w_k(x)|>R\}\Big| = 0,
\]
where the last relation comes from Chebyshev's inequality. Hence,
\begin{equation}\label{eq:10}
\int_Q g_k > -\eps + \int_{\{|\mathcal{B}w_k| < m\}\cap\{|\mathcal{B} b_k|< m\}} g_k,\mbox{ for all }\, m\geq m_\eps.
\end{equation}
By choosing $m_\eps$ larger if necessary, we also assume that
\begin{equation}\label{eq:11}
\left|\int_Q \langle \nu_x, D^2\tW(\bar{U})z\cdot z\,\mathbbm{1}_{\Rdd\setminus B(0,m)}(z)\rangle\right| < \eps,\mbox{ for all }\, m\geq m_\eps,
\end{equation}
where $\mathbbm{1}_{A}$ denotes the indicator function of a set $A\subset\Rdd$. Note that \eqref{eq:11} indeed holds true since $\int_Q \langle\nu_x, |z|^2\rangle < \infty$ and 
\begin{multline*}
\left|\int_Q \langle \nu_x, D^2\tW(\bar{U})z\cdot z\mathbbm{1}_{\Rdd\setminus B(0,m)}(z)\rangle\right| \leq C(W,K) \int_Q |\langle \nu_x, |z|^2\mathbbm{1}_{\Rdd\setminus B(0,m)}(z) \rangle|\\
= C(W,K) \int_Q |\langle \nu_x, |z|^2\rangle - \langle \nu_x, |\xi|^2\mathbbm{1}_{ B(0,m)}(z) \rangle|,
\end{multline*}
so \eqref{eq:11} follows by monotone and dominated convergence. In particular, \eqref{eq:11} says that for all $m\geq m_\eps$
\begin{equation}\label{eq:12}
\int_Q \langle \nu_x, D^2\tW(\bar{U})z\cdot z \rangle \leq \int_Q \langle \nu_x, D^2\tW(\bar{U})z\cdot z \mathbbm{1}_{B(0,m)}(z) \rangle + \eps.
\end{equation}
However, $\mathbbm{1}_{B(0,m)}$ is lower semicontinuous as the indicator function of the open ball $B(0,m)$. Thus, for all $x\in Q$ the function
\[
z\mapsto D^2\tW(\bar{U})z\cdot z \mathbbm{1}_{B(0,m)(z)}
\]
is lower semicontinuous and, since $(\mathcal{B}w_k)_k$ generates $(\nu_x)_{x\in Q}$, we infer that
\begin{multline}\label{eq:13}
\int_Q \langle \nu_x, D^2\tW(\bar{U})z\cdot z \mathbbm{1}_{B(0,m)}(z) \rangle \leq \liminf_{k\to\infty}\int_{\{|\mathcal{B}w_k|<m\}}D^2\tW(\bar{U})\mathcal{B}w_k\cdot\mathcal{B}w_k\\
= \liminf_{k\to\infty}\int_{\{|\mathcal{B}w_k|<m\}}D^2\tW(\bar{U}_k)\mathcal{B}w_k\cdot\mathcal{B}w_k.
\end{multline}
Indeed, the equality in \eqref{eq:13} follows from the fact that
\begin{multline*}
\int_{\{|\mathcal{B}w_k|<m\}}D^2\tW(\bar{U}_k)\mathcal{B}w_k\cdot\mathcal{B}w_k = \int_{\{|\mathcal{B}w_k|<m\}}D^2\tW(\bar{U})\mathcal{B}w_k\cdot\mathcal{B}w_k \\
+ \int_{\{|\mathcal{B}w_k|<m\}}\left[D^2\tW(\bar{U}_k) - D^2\tW(\bar{U})\right]\mathcal{B}w_k\cdot\mathcal{B}w_k
\end{multline*}
and that $\bar{U}_k \rightarrow \bar{U}$ in $C^0(Q)$. Combining \eqref{eq:13} with \eqref{eq:12}, for $m\geq m_\eps$,
\begin{multline}\label{eq:14}
\int_Q \langle \nu_x, D^2\tW(\bar{U})z\cdot z \rangle \leq \liminf_{k\to\infty}\int_{\{|\mathcal{B}w_k|<m\}}D^2\tW(\bar{U}_k)\mathcal{B}w_k\cdot\mathcal{B}w_k
+ \eps.
\end{multline}
To conclude the proof, we next claim that
\begin{multline}\label{eq:15}
\frac12 \liminf_{k\to\infty}\int_{\{|\mathcal{B}w_k|<m\}}D^2\tW(\bar{U}_k)\mathcal{B}w_k\cdot\mathcal{B}w_k
= \lim_{k\to\infty}\int_{\{|\mathcal{B}w_k|<m\}\cap \{|\mathcal{B} b_k|<m\}} g_k.
\end{multline}
Before proving \eqref{eq:15}, note that by \eqref{eq:14} and \eqref{eq:10}, it results in
\[
\frac12 \int_Q \langle \nu_x, D^2\tW(\bar{U})z\cdot z \rangle \leq \liminf_{k\to\infty}\int_Q g_k + \frac{3\eps}{2}.
\]
Taking $\eps\to 0$ we conclude \eqref{eq:8a} and Step 4.
We are left to prove \eqref{eq:15}. Recall that
\begin{align*}
&g_k  = \alpha_k^{-2}\left[\tW(\bar{U}_k+\alpha_k\mathcal{B} w_k|\bar{U}_k) - \tW(\bar{U}_k+\alpha_k\mathcal{B} b_k|\bar{U}_k)\right]\\
= & \int_0^1(1-s) \left[D^2\tW(\bar{U}_k+ s\alpha_k\mathcal{B}w_k)\mathcal{B}w_k\cdot\mathcal{B}w_k - D^2\tW(\bar{U}_k+ s\alpha_k\mathcal{B} b_k)\mathcal{B} b_k\cdot\mathcal{B}b_k\right].
\end{align*}
For convenience, let us write
\[
A_k :=\{x\in Q :|\mathcal{B}w_k(x)|<m\}\mbox{ and }B_k:=\{x\in Q :|\mathcal{B} b_k(x)|<m\}.
\]
Then, noting that $\int_0^1(1-s)\,ds = 1/2$
\begin{align*}
\mathbbm{1}_{A_k\cap B_k} g_k & = \mathbbm{1}_{A_k\cap B_k}\int_0^1(1-s) \left[D^2\tW(\bar{U}_k+ s\alpha_k\mathcal{B}w_k) - D^2\tW(\bar{U}_k)\right]\mathcal{B}w_k\cdot\mathcal{B}w_k\,ds\\
&\quad + \mathbbm{1}_{A_k} \frac12 D^2\tW(\bar{U}_k)\mathcal{B}w_k\cdot\mathcal{B}w_k - \mathbbm{1}_{A_k} \frac12 D^2\tW(\bar{U}_k)\mathcal{B}w_k\cdot\mathcal{B}w_k \left(1 - \mathbbm{1}_{B_k}\right)\\
&\quad - \mathbbm{1}_{A_k\cap B_k} \int_0^1(1-s) D^2\tW(\bar{U}_k+ s\alpha_k\mathcal{B} b_k)\mathcal{B} b_k\cdot\mathcal{B} b_k \,ds\\
& =: I_1^k + I_2^k + I_3^k + I_4^k.
\end{align*}
We immediately infer that
\[
\int_Q I_2^k = \frac12 \int_{\{|\mathcal{B}w_k|<m\}}D^2\tW(\bar{U}_k)\mathcal{B}w_k\cdot\mathcal{B}w_k
\]
and in order to conclude to \eqref{eq:15} we show that 
\[
\lim_{k\to\infty} \int_Q I_1^k = \lim_{k\to\infty} \int_Q I_3^k = \lim_{k\to\infty} \int_Q I_4^k = 0.
\]
Recall that $\alpha_k\to 0$ and $\bar{U}_k \to \bar{U}$ in $C^0(Q)$. Thus, for $I_1^k$ and since we are in the set $A_k$, we find that
\[
\left|D^2\tW(\bar{U}_k+s\alpha_k\mathcal{B}w_k) - D^2\tW(\bar{U}_k)\right| \leq C(W,K) \alpha_k m^3 \to 0,\quad k\to\infty.
\]
Thus, $\int_Q I_1^k \to 0$ by dominated convergence. As for $ I_3^k$, again since $D^2\tW$ is continuous and $\|\bar{U}_k\|_{L^\infty(Q)}\leq K$, we get that
\[
|I_3^k| \leq C(W,K) m^2 \left(1 - \mathbbm{1}_{ \{ |\mathcal{B} b_k|<m\}}\right) = C(W,K) m^2  \mathbbm{1}_{\{|\mathcal{B} b_k|\geq m\}}.
\]
Hence, $\int_Q I_3^k \to 0$
as $\mathcal{B} b_k \to 0$ in measure. Lastly, for $I_4^k$, as we are in $B_k$ and $s\in(0,1)$, we get that $\bar{U}_k + s\alpha_k \mathcal{B} b_k \rightarrow \bar{U}$ uniformly as $k\to\infty$ and thus
\[
|I_4^k| \leq C(W,K)m |\mathcal{B} b_k| \to 0\mbox{ in measure}.
\]
In particular, restricting to $B_k$, $\int_Q I_4^k \to 0$ by dominated convergence. This concludes the proof of Step 4.\\\quad\\
\underline{Step 5}: We show how \eqref{eq:9} leads to a contradiction. By Lemma \ref{lemma:tWall} (2)
\[
f(x,z) := D^2\tW(\bar{U}(x))z\cdot z 
\]
is  $\mathcal{A}$-quasiconvex for each $x\in Q$. Since $(\mathcal{B}w_k)_k$ generates the $\mathcal{A}$-2-Young measure $(\nu_x)_{x\in Q}$ and $f(x,z)$ grows quadratically in $z$, Jensen's inequality for $\mathcal{A}$-quasiconvex functions \cite[Theorem 4.1.]{Fon} says that for a.e. $x\in Q$
\begin{align*}
D^2\tW(\bar{U}(x))\mathcal{B}w\cdot\mathcal{B}w = f(x,\langle\nu_x,{\rm id}\rangle)
 \leq \langle \nu_x, f(x,\cdot)\rangle = \langle\nu_x, D^2\tW(\bar{U}(x))z\cdot z\rangle.
\end{align*}
Adding $c_1\sum_{i=1}^l|\nabla^{l-i}w|^2$ on both sides and integrating over $Q$, we infer that
\begin{multline*}
c_1\sum_{i=1}^l\int_Q|\nabla^{l-i}w|^2 +\int_Q  D^2\tW(\bar{U}(x))\mathcal{B}w\cdot\mathcal{B}w  \\
\leq c_1\sum_{i=1}^l\int_Q|\nabla^{l-i}w|^2 + \langle\nu_x, D^2\tW(\bar{U}(x))z\cdot z\rangle \leq 0,
\end{multline*}
by \eqref{eq:9}. However, by Proposition \ref{prop:1}, since $w\in W^{l,p}(Q)$, we know that
\[
c_1\sum_{i=1}^l\int_Q|\nabla^{l-i}w|^2 + \int_Q D^2\tW(\bar{U}(x))\mathcal{B}w\cdot\mathcal{B}w \geq \frac{c_0}2\int_Q|\mathcal{B}w|^2,
\]
and, hence, $\mathcal{B}w = 0$ and $w= \mathcal{F}^{-1}(\mathcal{B}^{\dagger}(\cdot))\star\mathcal{B}w= 0$.
We may thus apply Proposition \ref{prop:2} with $a_k = \alpha_k$ and $h_k = \alpha_k w_k$. Recall that $\alpha_k^{p-2} = \Lambda \eta_k^p$,
where, by Step 1, $\Lambda = \beta_k^p/\alpha_k^2$ is bounded. Thus,
\[
\alpha_k^{-2}|V(\alpha_k \nabla^{l-i}w_k)|^2 = |\nabla^{l-i}w_k|^2 + \Lambda |\eta_k \nabla^{l-i}w_k|^p \rightarrow 0\mbox{ in }L^1(Q),
\]
 for i=1,..,$l$. Also, $\alpha_k^{-2}|V(\alpha_k \mathcal{B} w_k)|^2 = |\mathcal{B} w_k|^2 + \Lambda |\eta_k \mathcal{B} w_k|^p$ is bounded in $L^1(Q)$. So, recalling that $\alpha_k w_k = \varphi_k$, Proposition \ref{prop:2} says that
\begin{align*}
0 < \frac{c_0}4 & = \liminf_{k\to\infty} \frac{c_0}4 \int_Q |\mathcal{B}w_k|^2 \\
& \leq \liminf_{k\to\infty} \frac{c_0}4 \int_Q |\mathcal{B}w_k|^2 + \alpha_k^{p-2}|\mathcal{B}w_k|^p\\
& \leq \liminf_{k\to\infty}\alpha_k^{-2}\int_Q\tW(\bar{U}_k+\alpha_k\mathcal{B} w_k|\bar{U}_k)\\
& = \liminf_{k\to\infty}\alpha_k^{-2}\int_Q \tW(\bar{U}_k+\mathcal{B} \varphi_k|\bar{U}_k) + \frac{c_1}{2} \sum_{i=1}^l\int_Q|\nabla^{l-i}w_k|^2\leq 0,
\end{align*}
by \eqref{eq:6}. But $c_0>0$, concluding the proof of Theorem \ref{thm:L^2small}.
\end{proof}
\section{An application in dynamics: local stability and weak-strong uniqueness}\label{sec:involutions}

In this section, we study local stability and weak-strong uniqueness properties for general systems of conservation laws \eqref{eq:scl} possessing involutions \eqref{eq:involutions} and an $\mathcal{A}$-quasiconvex entropy. In particular, for $T>0$ and $Q=(0,1)^d$, we examine the system
\begin{equation}
  \label{eq:scl}
  \begin{aligned}
   	\partial_t U(t,x) + {\rm div}_x f(U(t,x))&=0,\quad(t,x)\in (0,T)\times Q  \\ U(0,x)&=U^0(x),\quad x\in Q
   \end{aligned}
\end{equation} 
for the unknown $Q$-periodic function $U:(0,T)\times Q\to\mathbb{R}^N$ with
\begin{equation}
\int_{Q}U(t,x)\;dx=0, \mbox{ for all } 0<t\leq T.
\end{equation}
In \eqref{eq:scl}, the flux function $f=(f_{i\al})_{(i,\al)\in\mathbb{R}^{N\times d}}:\mathbb{R}^N\to\mathbb{R}^{d\times N}$ is a given locally Lipschitz mapping.
We say that system \eqref{eq:scl} possesses an involution if there exists a linear differential operator $\mathcal{A}$ with the property that
\begin{equation}
\label{eq:involutions}
  \mathcal{A} U^0=0\,\,\Rightarrow\,\,\mathcal{A} U(t,\cdot)=0\quad\text{for all }t\in (0,T).
\end{equation}
Typical examples include the equations of elasticity and electromagnetism, see \cite{DafermosB}. Indeed, the equations of motion of a hyperelastic body in the absence of external forces take the form $y_{tt} = {\rm div}DW(\nabla y)$
where $W$ denotes the stored energy function. Upon the change of variables $v = y_t$ and $F = \nabla y$, we obtain the system
 \begin{align*}
   	\partial_t v - {\rm div}_x DW(F)&=0, \\
   	\partial_t F - \nabla v&=0,\\
   	{\rm curl}\,F&=0. \label{eq:nle}
\end{align*}
The second equation shows that $\mathcal{A} = {\rm curl}$ is an involution. Similarly, in linear elasticity, the equations take the form
 \begin{align*}
   	\partial_t u - {\rm div}_x C E&=0, \\
   	\partial_t E -  \mathcal{E}(u)&=0, \\
   	{\rm curl}\,{\rm curl}\, E &= 0,
\end{align*}
where $2 \mathcal{E}(u) = \nabla u + (\nabla u)^T$ and $\mathcal{A} = {\rm curl\,curl}$ is an involution whose kernel consists of symmetric gradients. Also, note that a natural assumption on the quadratic form $CE:E$ is convexity on the wave cone of the operator ${\rm curl\,curl}$ which, by Lemma \ref{lemma:quadratic2}, is equivalent to ${\rm curl\,curl}$-quasiconvexity. Moreover, the equations of electromagnetism in the absence of charges and currents become
\begin{align*}
   	\partial_t B + {\rm curl} E&=0, \\
   	\partial_t D - {\rm curl} H&=0,\\
   	{\rm div}\,B={\rm div}\,D&=0, 
\end{align*}
where $B$ is the magnetic induction, $D$ is the electric displacement, and $E$, $H$ are, respectively, the electric and magnetic fields. Typically, Maxwell's equations are assumed linear, however, there are relevant nonlinear theories, see \cite{Coleman}, \cite{Serre}, \cite{DafermosB}, with the so-called Maxwell's equations in the Born-Infeld medium being the most known. The reader is referred to \cite{Brenier} for a mathematical treatment.

Note that in continuum mechanics, systems like \eqref{eq:scl}, are typically supplemented with an inequality of the form
\begin{equation}\label{ecl}
    \partial_t \eta+ {\rm div}_x q\leq 0,
\end{equation}
known as the Clausius-Duhem inequality, expressing the second law of thermodynamics in this context. Mathematically, $\eta:\mathbb{R}^N\to\R$ is referred to as an entropy and $q:\R^N\to \R^d$ as an entropy flux and are assumed to satisfy
\begin{equation}\label{eq:etaq1}
\frac{\partial q_{\al}}{\partial U_i}=\frac{\partial\eta}{\partial U_j}\frac{\partial f_{j\al}}{\partial U_i}.
\end{equation}
In particular,  
\begin{equation}\label{eq:etaq2}
    \frac{\partial^2\eta}{\partial U_k\partial U_j}\frac{\partial f_{j\al}}{\partial U_i}=\frac{\partial^2\eta}{\partial U_i\partial U_j}\frac{\partial f_{j\al}}{\partial U_k}
\end{equation}
and thus Lipschitz solutions to \eqref{eq:scl} satisfy \eqref{ecl} as an equality. 

Entropies in physical systems are often convex which, combined with \eqref{eq:etaq2}, renders the system symmetrisable upon the change of variables $U\to D\eta(U)$ and hence locally well posed, see \cite{DafermosB}. At the same time, inequality \eqref{ecl} restricts admissible solutions and may rule out unphysical solutions.   

On the other hand, it is also known that convexity of the entropy may be ruled out as a consequence of physical invariance. This is precisely the case in nonlinear elasticity due to frame-indifference \cite{DafermosB}, and in electromagnetism due to Lorentz invariance \cite{Serre}. However, the presence of involutions may compensate this loss of convexity, but only in the directions where the operator $\mathcal{A}$ has elliptic behaviour. Essentially, the ``bad'' behaviour is expected to occur in the directions of the wave cone $\Lambda_\mathcal{A}$, and convexity along these directions, i.e. $\Lambda_\mathcal{A}$-convexity, may be enough to partially recover results ensured by convexity. 

Indeed, Dafermos in \cite{Daf} examined such systems endowed with a $\Lambda_\mathcal{A}$-convex entropy and, under additional assumptions on the involutions $\mathcal{A}$, recovered hyperbolicity. Moreover, he showed that local stability and weak-strong uniqueness results can also be recovered within a class of $BV$ weak solutions, if they satisfy an assumption of \emph{small local oscillations}, required to prove a G\aa rding-type inequality for $\Lambda_{\mathcal{A}}$-convex functions. In this section, we show that in fact this assumption is redundant when the entropy is $\mathcal{A}$-quasiconvex. In this sense, $\mathcal{A}$-quasiconvexity captures the structure of these systems and arises as a natural convexity condition.

We note that Maxwell's equations do not generally fall under this setting. For vector fields $B,\,D:\R^3\to\R^3$, the wave cone of $\mathcal{A}={\rm div}$ is the entire space $\R^6$ and thus $\mathcal{A}$-quasiconvexity and $\Lambda_{\mathcal{A}}$-convexity reduce to convexity. However, when $B,\,D:\R^2\to\R^3$, the wave cone is strictly smaller than $\R^6$.
Still, it is a matter of tedious computations to show that the entropy at least for the Born-Infeld medium is not even $\Lambda_{\mathcal{A}}$-convex and thus, unlike polyconvex elasticity, not convex in the null-Lagrangians of $\mathcal{A} = {\rm div}$. Nevertheless, similar to polyconvex elasticity, the system can be extended to an enlarged system that admits a convex entropy, see \cite{DafermosB,Serre}.

In the sequel, we assume that an entropy-entropy flux pair exists satisfying \eqref{eq:etaq1} and that $\eta$ satisfies the assumptions (H1)-(H4). Moreover, as in \cite{Daf}, we assume that weak solutions are bounded. We refer the reader to Remark \ref{remark:growths} following the proof for a discussion on these assumptions.

\noindent\begin{definition}\label{def1} 
Let $U\in L^\infty((0,T)\times Q)$. We say that the function $U$ is a dissipative weak solution to \eqref{eq:scl} with initial data $U^0$ if
  \begin{equation}\label{eq:weak}
      \int_Q \phi_i(0,\cdot)U^0_i+\int_0^T\int_Q \partial_t\phi_i\cdot U_i  + \int_0^T\int_Q \partial_{\al}\phi_i \cdot f_{i\al}(U) = 0
  \end{equation}
  for any $\phi\in C^1_c([0,T),C^1(Q))$ and $i$=1,..,$N$, and the dissipation inequality
   \begin{equation}\label{eqWE}
   \int_Q\theta(0)\eta(U^0)+\int_0^T\int_Q\dot{\theta}\;\eta(U) \geq 0
   \end{equation}
 holds for any nonnegative test function $\theta\in C^1_c([0,T))$.
\end{definition}

Recall that Lipschitz solutions $\bar{U}\in W^{1,\infty}([0,T]\times \overline{Q})$ satisfy \eqref{eqWE} as an equality, that is
\begin{equation}\label{eqSE}
   \int_Q\theta(0)\eta(\bar{U}^0) +\int_0^T\int_Q\dot{\theta}\;\eta(\bar{U})= 0.
   \end{equation}
Moreover, note that if $\int_Q U^0 = 0$ then also $\int_QU(t,\cdot) = 0$ for a.e. $t\in(0,T)$. This follows by testing \eqref{eq:weak} with $\phi(t,x) = \theta(t)$ where $\theta$ localises at a fixed time, as in \eqref{eq:localise}.
The main theorem of this section now follows, cf. \cite[Theorem 4.1]{Daf}.

\begin{theorem}\label{wsu}
Let $\bar{U}\in W^{1,\infty}([0,T]\times\overline{Q})$ and $U\in L^\infty((0,T)\times Q)$ be, respectively, a strong and a dissipative  weak solution of \eqref{eq:scl} emanating from the zero-average initial data $\bar{U}^0$, $U^0 \in L^\infty(Q)$. Assume that $U$ and $\bar{U}$ satisfy the PDE constraint $\mathcal{A}\bar{U} = \mathcal{A}U = 0$, and that the entropy $\eta$ satisfies (H1)-(H4). Then, there exist constants $C_1,\,C_2>0$
such that for almost all $t \in (0, T)$
\[
\int_Q|V(U(t,\cdot)-\bar{U}(t,\cdot))|^2 \leq C_1\int_Q |V(U^0-\bar{U}^0)|^2 \, e^{C_2\,t},
\]
where $V$ is the auxiliary function defined in \eqref{eq:auxf}.
\end{theorem}

\begin{proof}
Let $U$ and $\bar{U}$ as in the statement and test the equations \eqref{eq:weak} with the function $\phi(t,x) = \theta(t)D\eta(\bar{U}(t,x))$, where $\theta\in C^1_c([0,T))$. Note that this is an appropriate test function by density. Subtracting the equations for $U$ from the equations for $\bar{U}$, we infer that 
\begin{align*}
&\int_Q \theta(0)\,D_j\eta(\bar{U}^0)\,(U^0_j-\bar{U}^0_j) +\int_0^T\int_Q \dot{\theta}\,D_j\eta(\bar{U})\,(U_j-\bar{U}_j) \\
& =-\int_0^T\int_Q\theta\left\{\,\partial_{\al}D_k\eta(\bar{U})\,\big(f_{k\al}(U)-f_{k\al}(\bar {U})\big) + \partial_tD_j\eta(\bar{U})\,( U_j-\bar{U}_j)\right\},
\end{align*}
where $\partial_\al,\,\partial_t$ and $D_j$ stand for the operators $\frac{\partial}{\partial x_\al},\,\frac{\partial}{\partial t}$ and $\frac{\partial}{\partial U_j}$ respectively. By \eqref{eq:etaq2}, we observe that $\partial_tD_j\eta(\bar{U})= -\partial_{\al}D_k$
and thus
\begin{equation}
\label{eq31}
\begin{aligned}
&\int_Q \theta(0)\,D_j\eta(\bar{U}^0)\,(U^0_j-\bar{U}^0_j) +\int_0^T\int_Q \dot{\theta}\,D_j\eta(\bar{U})\,(U_j-\bar{U}_j) \\
&=-\int_0^T\int_Q \theta\,\left[ \partial_{\al}D_k\eta(\bar{U})\right]\,f_{k\al}(U|\bar{U}) =:\mathcal{R},
\end{aligned}
\end{equation}
\noindent
where $f_{k\al}(U|\bar{U}):=f_{k\al}(U)-f_{k\al}(\bar{U})-D_j f_{k\al}(\bar{U})\,(U_j-\bar{U}_j)$ is the relative flux. This complies with the notation in the previous section as $U = \bar{U} + (U - \bar{U})$. Next, by \eqref{eqWE}, \eqref{eqSE} and \eqref{eq31}, we get that
\begin{equation}
\label{eq:eq32}
\int_0^T\int_Q \dot{\theta}\,\eta(U|\bar{U}) +\int_Q \theta(0)\,\eta(U^0|\bar{U}^0)\geq -\mathcal{R},
\end{equation}
where the relative entropy is given by
\[
\eta(U|\bar{U}) = \eta(U) - \eta(\bar{U}) - D_j\eta(\bar{U})(U_j - \bar{U}_j).
\]
Indeed, \eqref{eq:eq32} follows by \eqref{eqWE} and \eqref{eqSE} since
\begin{align*}
   &\int_0^T\int_Q \dot{\theta}\,\eta(U|\bar{U})+\int_Q \theta(0)\,\eta(U^0|\bar{U}^0) \\
   =&\int_0^T\int_Q \dot{\theta}\,\eta(U) +\int_Q \theta(0)\,\eta(U^0) 
   -\int_0^T\int_Q \dot{\theta}\,\eta(\bar{U})-\int_Q \theta(0)\,\eta(\bar{U}^0) \\
   -&\int_0^T\int_Q \dot{\theta}\,D_j\eta(\bar{U})\,(U_j-\bar{U}_j)\,-\int_Q \theta(0)\,D_j\eta(\bar{U}^0)\,(U^0_j -\bar{U}^0_j).
\end{align*}
We next follow a standard argument to localise in time by letting $(\theta_m)_{m\in\mathbb{N}}\subset C^{\infty}_c([0,T))$ be a bounded sequence approximating the function
\begin{equation}
\label{eq:localise}
     \theta(\tau)= 
      \left\{
          \begin{array}{cc}
        
               1,  & \tau\in [0,t) \\
               (t-\tau)/\epsilon+1 ,  & \tau\in [t,t+\epsilon) \\
               0,     &\tau\in[t+\epsilon,T)
             
          \end{array}
      \right.
\end{equation}
such that $(\theta_m)_m$ is nonincreasing and $\dot{\theta}_m(\tau)\to\dot{\theta}(\tau)$ for all $\tau\neq t,t+\epsilon$. Note that $\dot{\theta}_m\leq 0$ and so testing
\eqref{eq:eq32} with $\theta_m$ we find that
\begin{equation}
\label{eq33}
\begin{aligned}
&\int_0^T\int_Q |\dot{\theta}_m(\tau)| \,\eta(U(\tau,x)|\bar{U}(\tau,x))\,dxd\tau  \leq \mathcal{R}+\int_Q \eta(U^0(x)|\bar{U}^0(x))\,dx.
\end{aligned}
\end{equation}
Since $U\in L^\infty((0,T)\times Q)$, $f_{k\alpha}$ is locally Lipschitz and $\partial_{\al}D_k\eta(\bar{U})$ is bounded, we compute from \eqref{eq31} that
\[
|\mathcal{R}|\leq C\int_0^T\int_Q |\theta|\,|U-\bar{U}|^2.
\]
As $U$ is bounded, taking the limit $m\to\infty$ in \eqref{eq33} by dominated convergence, gives
\begin{equation*}
\frac{1}{\epsilon}\int_t^{t+\epsilon}\int_Q \eta(U|\bar{U}) \leq C\int_0^{t+\epsilon}\int_Q |U-\bar{U}|^2 +\int_Q \eta(U^0|\bar{U}^0).
\end{equation*}
Then, sending $\epsilon\to 0$, we get that for almost all $t\in (0,T)$,
\begin{equation*}
\int_Q \eta(U|\bar{U})\leq C\int_0^t\int_Q|U-\bar{U}|^2 + \int_Q \eta(U^0|\bar{U}^0).
\end{equation*}
Note that the relative entropy is quadratic on bounded functions and thus, by Theorem \ref{theorem:2}, we deduce that for almost all $t\in (0,T)$ and up to a suitable constant 
\begin{equation}
\label{eq35}
\begin{aligned}
 \int_Q|V(U-\bar{U})|^2  \lesssim \int_0^t\int_Q |U-\bar{U}|^2   + \int_Q |V(U^0-\bar{U}^0)|^2 + \|U-\bar{U}\|_{W^{-1,(2,p)}}^2,  
 \end{aligned}
\end{equation}
where $\|\cdot\|_{W^{-1,(2,p)}}$ is the auxiliary mapping defined in \eqref{eq:auxpr}. In order to apply Gr\"{o}nwall's inequality and conclude the proof, it remains to estimate the last term on the right-hand side of \eqref{eq35}. Similarly to Dafermos in \cite{Daf}, for $r\in\{2,p\}$, we infer that since $L^r(Q)$ embeds into $W^{-1,r}(Q)$
\[
\|U(t,\cdot)-\bar{U}(t,\cdot)\|_{W^{-1,r}}\lesssim \|U^0-\bar{U}^0\|_{L^{r}}+\int_0^t\|\partial_t\{U(s,\cdot)-\bar{U}(s,\cdot)\}\|_{{W^{-1,r}}}ds.
\]
 By taking into account \eqref{eq:scl} we deduce the bound
\begin{align}
    \|\partial_t\{U(s,\cdot)-\bar{U}(s,\cdot)\}\|_{W^{-1,r}(Q)}&\leq  \|\partial_\al f_{i\al}(U)-\partial_\al f_{i\al}(\bar{U})\|_{W^{-1,r}(Q)}\nonumber \\
    &\leq C \| f_{i\al}(U)- f_{i\al}(\bar{U})\|_{L^{r}(Q)}\nonumber \\
    &\leq C \|U(s,\cdot)-\bar{U}(s,\cdot)\|_{L^{r}(Q)},\label{eq:lip_flux}
\end{align}
where the last inequality follows from the fact that $f$ is locally Lipschitz and $U$ is bounded. Finally, by H\"older's inequality, we infer that
\[
\|U(t,\cdot)-\bar{U}(t,\cdot)\|_{W^{-1,r}}^r\lesssim \|U^0 - \bar{U}^0\|_{L^{r}}^r + \,T^{\frac{r}{r-1}}\int_0^t\|U(s,\cdot)-\bar{U}(s,\cdot)\|^r_{L^{r}}ds.
\]
Returning to \eqref{eq35} and applying the above bound for $r=2$ and $r=p$ we arrive at
\[
 \int_Q|V(U-\bar{U})|^2  \lesssim \int_0^t\int_Q |V(U-\bar{U})|^2 + \int_Q |V(U^0 - \bar{U}^0)|^2.   
\]
An application of Gr\"{o}nwall's inequality completes the proof.
\end{proof}

\begin{remark}\label{remark:growths}
Note that the $L^\infty$ bounds on weak solutions are needed in the estimate \eqref{eq:lip_flux}. Otherwise, mild growth assumptions on the flux suffice to consider merely $L^p$ weak solutions. 
Moreover, we note that the assumed growths on $\eta$ do not e.g. directly apply to elasticity where $\eta(v,F) = \frac12 |v|^2 + W(F)$. However, as $|v|^2$ is convex, it is  immediate to deduce the result assuming (H1)-(H4) on $W$ \cite{Spirito}.
\end{remark}

\begin{remark}\label{rem:elliptic}
{\bf ($L^p$ bounds and elliptic estimates)} We propose a general structure that allows us to recover elliptic estimates, similar to those in \cite{Spirito}, for merely $L^p$ solutions of system \eqref{eq:scl}. To be more precise, instead of the PDE constraint \eqref{eq:involutions} we assume that
\begin{equation}\label{eq:extra}
\partial_tC(U(t,x))+\mathcal{B}g(U(t,x))=0,
\end{equation}
where $\mathcal{B}$ is a potential operator of $\mathcal{A}$, $g:\mathbb{R}^N\to\mathbb{R}^N$ is globally Lipschitz and $C:\mathbb{R}^N\to\mathbb{R}^N$ such that $\left(C(U)\right)_i\in \left\{0,U_i\right\}$ for $i=1,..,N$. In particular, the non-zero rows of $C(U)$ constitute the constrained components of $U$. 

In this setting the involutions \eqref{eq:involutions} are embodied in \eqref{eq:extra} which may thus serve as an alternative formulation. The latter equation may seem restrictive but it is satisfied in the equations of elasticity and electromagnetism for $(\mathcal{B},\mathcal{A})=(\nabla,\rm curl)$ and $(\mathcal{B},\mathcal{A})=(\rm curl,\rm div)$ respectively.

Suppose in addition that $\mathcal{B}$ is first-order and elliptic which is true in elasticity but not in Maxwell's equations. Note that the estimates of Lemma \ref{Ptype} are now a consequence of ellipticity and in particular of the fact that $\mathbb{B}^*(\xi)\mathbb{B}(\xi)$ is invertible for all $\xi\in\mathbb{R}^d\setminus\{0\}$. We claim that these assumptions suffice to bound the term $\|U-\bar{U}\|_{W^{-1,(2,p)}}^2$ and replace estimate \eqref{eq:lip_flux} without any $L^\infty$ assumptions. 

Below, we sketch the proof of this claim for the simpler case $p=2$ and $C=id$ although the general case follows similarly. For zero-average $U\in L^\infty(0,T;L^2(Q))$ and $W\in L^\infty(0,T;W^{1,2}(Q))$ a primitive of $U$, equation \eqref{eq:extra} implies that
\begin{equation*}
    \int_0^T\int_Q(W-\bar{W})\mathcal{B}^*\psi_t-\int_0^T\int_Q\left(g(U)-g(\bar{U}) \right){B}^*\psi=0,
\end{equation*}
for all $\psi\in C^\infty_c([0,T);C^\infty(Q))$. Now, by testing the above equation with $\psi=\mathcal{B}h$ where, for $\phi\in C^\infty_c([0,T);C^\infty(Q))$ with zero average, $h$ is the unique solution of the elliptic system
\begin{align*}
    -\mathcal{B}^*\mathcal{B}h&=\phi,\quad\int_Qh =0,
\end{align*}
we infer that
\begin{equation}\label{eq:eest}
    \int_0^T\int_Q(W-\bar{W})_t\phi-\int_0^T\int_Q\left(g(U)-g(\bar{U})\right)\phi=0.
\end{equation}
Note that we have moved the time derivative on $(W-\bar W)$. This is indeed possible since by \eqref{eq:extra} and the fact that $g$ is globally Lipschitz, $U_t \in L^\infty(0,T;H^{-1}(Q))$. In particular, $\mathcal{B} W_t \in L^\infty(0,T;H^{-1}(Q))$ and by ellipticity of $\mathcal{B}$, we infer that $W_t \in L^\infty(0,T;L^2(Q))$.
We may now test \eqref{eq:eest} with the function $\phi=W-\bar{W}$, while localising in time, to get that by Young's inequality and the Lipschitz condition on $g$,
\begin{equation*}
    \int_Q|W-\bar{W}|^2\lesssim \int_Q |W^0 - \bar{W}^0|^2 + \int_0^t\int_Q|U-\bar{U}|^2+\int_0^t\int_Q|W-\bar{W}|^2,\,\, t\in (0,T).
\end{equation*}
Then, the above estimate inserted in \eqref{eq35} and Gr\"{o}nwall's inequality allows us to complete the proof. In the case $p > 2$, one may follow a strategy as in \cite{Spirito} where the Sobolev inequalities arise from the ellipticity of $\mathcal{B}$.  Moreover, when $C(U)\neq U$, additional structure in the PDE is required to conclude the relative entropy argument as in the case of elasticity.
\end{remark}

\section{An application in statics: sufficient conditions for local minimisers}\label{sec:min}

In this section, we study functionals of the form
\begin{align}\label{eq:mp}
\mathcal{W}[U]:=\int_Q W(U(x))dx, 
\end{align}
for $U \in L^p_{\mathcal{A}}(Q)$ where 
\[
L^p_{\mathcal{A}}(Q):=\left\{U\in L^p(Q)\,:\,\mathcal{A}U = 0,\,\int_Q U = 0\right\}.
\]
Motivated by recent developments in the vectorial Weierstrass problem \cite{JudithR, JudithKostas, GM09}, we provide an appropriate generalisation for functionals of the form \eqref{eq:mp} and differential operators other than ${\rm curl}$, that is we establish sufficient conditions for local minimisers in the strong $W^{-1,p}$ topology based on $\mathcal{A}$-quasiconvexity assumptions. We remark that the presented result entails a quantitative version of uniqueness for these minimisers, see also Corollary \ref{cor:unique}, which had not been previously observed. The proof comes as a direct consequence of Theorem \ref{thm:L^2small} which formed the basis for the G\aa rding inequality and its proof has been largely motivated by these recent developments on the Weierstrass problem. 

In particular, we show the following theorem. We note that the natural space of variations for $\mathcal{W}$ is given by
\[
\left\{\psi \in C(Q)\,:\,\mathcal{A}\psi = 0,\,\int_Q\psi = 0\right\}.
\]
However, under the growth assumptions (H3), one may equivalently consider the closure of variations in $L^p$ given by the space $L^p_{\mathcal{A}}(Q)$.

\begin{theorem}\label{thm:min}
Assume that $W\in C^2(\R^N)$ satisfies (H3), (H4) and let $\bar{U}\in  L^p_\mathcal{A}(Q)\cap C(\overline{Q})$ such that the following conditions hold:
\begin{itemize}
    \item $\bar U$ is a weak solution of the Euler-Lagrange equations, $\mathcal{B}^* DW(\bar U) = 0$, i.e.
\[
\int_Q DW(\bar{U}(x))\psi(x) dx=0,
\]
for all $\psi\in  L^p_\mathcal{A}(Q)$;
\item the second variation is strongly positive at $\bar U$, i.e.
\[
\int_QD^2W(\bar{U}(x))\psi(x)\cdot\psi(x)dx\geq c\int_Q |\psi(x)|^2dx,
\]
for all $\psi\in  L^p_\mathcal{A}(Q)$;
\item $W$ is strongly $\mathcal{A}$-quasiconvex at $\bar U(x_0)$ for all $x_0\in Q$, i.e. 
\[
\int_Q \left[W(\bar U(x_0) + \psi(x)) - W(\bar U(x_0))\right] dx \geq c_0 \int_Q |V(\psi(x))|^2dx,
\]
for all $\psi\in  L^p_\mathcal{A}(Q)$.
\end{itemize}
Then, there exists $\eps_0>0$ and $C>0$ such that 
\begin{align*}
    \mathcal{W}[U] - \mathcal{W}[\bar U]\geq C\int_Q|V(U(x)-\bar{U}(x))|^2dx,
\end{align*}
for all $U \in  L^p_\mathcal{A}(Q)$ with $\|U - \bar{U}\|_{W^{-1,p}(Q)}\leq \varepsilon_0$.
\end{theorem}

\begin{proof}
The main ingredient in the proof is Theorem \ref{thm:L^2small} combined with the simple observation that if $\bar U$ solves the Euler-Lagrange system, then
\[
\int_Q W(\bar U + \psi|\bar U) = \int_Q \left[W(\bar U +\psi) - W(\bar U)\right] = \mathcal{W}(\bar U +\psi) - \mathcal{W}(\bar U),
\]
for any $\psi \in L^p_{\mathcal{A}}(Q)$. Note that the relative energy $W(\cdot|\cdot)$ is precisely the so-called Weierstrass excess or E-function for the functional $\mathcal{W}$. Thus, given $\bar U$ as in the statement, let $U\in L^p_{\mathcal{A}}(Q)$ and set $\psi = U - \bar U \in L^p_{\mathcal{A}}(Q)$.
We prove that there exists $\eps_0>0$ and $C>0$ such that 
\begin{align*}
    \int_Q W(\bar U + \psi|\bar U)\geq C\int_Q|V(\psi)|^2
\end{align*}
whenever $\|\psi\|_{W^{-1,p}(Q)} \leq \varepsilon_0$. 
This is precisely the statement of Theorem \ref{thm:L^2small} without the penalty term $\|\psi\|^2_{W^{-1,(2,p)}}$. One may now proceed in the exact same way as in the proof of Theorem \ref{thm:L^2small}, without the penalty term, noting that this is only required in Step 5 where Proposition \ref{prop:1} is applied. 

In the present case, we claim that the strong positivity of the second variation of $W$ at $\bar U$ implies the strong positivity of the second variation of $\tilde W$ at $\bar U$ which replaces the need for Proposition \ref{prop:1} in Step 5. 
Indeed, below we show that 
\begin{equation}\label{eq:2ndVarlast}
\int_Q D^2\tW(\bar U)\psi \cdot \psi \gtrsim \int_Q|\psi|^2
\end{equation}
which assumes the role of Proposition \ref{prop:1} in our setting. To prove \eqref{eq:2ndVarlast}, note that $\tilde W$ is defined in Lemma \ref{lemma:tWall} as $\tilde W(\lambda) = W(\lambda) - c_2|V(\lambda)|^2$
where $c_2=c_2(W,K)$ can be chosen even smaller if necessary. For $|\lambda|\leq K$ and $z\in \R^N$, we compute that $|D^2 \left(|V(\lambda)|^2\right) z \cdot z| \leq C(p,K)|z|^2$.
and we may thus choose $c_2 = c_2(p,K)$ small enough so that for $\|\bar U\|_{L^\infty}\leq K$ and $\psi \in L^p_{\mathcal{A}}(Q)$, \eqref{eq:2ndVarlast} holds. This completes the proof.
\end{proof}

\begin{remark}
Note that in the case $\mathcal{A} = {\rm curl}$, Theorem \ref{thm:min} reduces to a statement about $L^p$ local minimisers, thus recovering partially the result in \cite{JudithR}. In fact this is a statement about $L^p$ local minimisers for any operator $\mathcal{A}$ that admits an elliptic, first-order potential $\mathcal{B}$. Indeed, ellipticity is required to control the $L^p$ norm of the primitive by the $W^{-1,p}$ norm of the function without reverting to properties of the potential operator as in Lemma \ref{Ptype}.

We also remark that extending the presented result to the case of a bounded domain $\Omega$ is nontrivial as, working on the torus, allows for Fourier Analysis tools that are otherwise not available. However, for $\mathcal{A} = {\rm curl}$, the above result can be extended in a straightforward way for pure displacement boundary conditions. In fact, with slight modifications one may treat problems with mixed boundary conditions, whereby a part of the boundary remains free. Then, one needs to append the sufficient conditions of Theorem \ref{thm:min} with quasiconvexity at the boundary, see \cite{JudithR} as well as \cite{JudithKostas, GM09} for $L^\infty$ local minimisers. Below we show that this is indeed true in the form of a corollary that extends existing results to include a quantitative estimate of uniqueness. The case of functionals depending on lower order terms and $L^\infty$ local minimisers lies outside the scope of the present work. We refer the reader to \cite{JudithR, JudithKostas, GM09} for discussions on quasiconvexity at the boundary. Note that a notion of $\mathcal{A}$-quasiconvexity at the boundary for $p$-homogeneous functions was defined in \cite{Aqcb} in the context of lower semicontinuity for signed integrands.  
\end{remark}

For the following corollary, let $\Omega\subset \R^d$ a bounded domain with $C^1$ boundary $\partial\Omega$ such that
\[
\partial\Omega = \Gamma_D \cap \Gamma_N
\]
where $\Gamma_D$ is a relatively open subset of $\partial\Omega$ and $\Gamma_N = \partial\Omega\setminus \overline{\Gamma_D}$, where $\Gamma_D$ is the relative interior of $\overline{\Gamma_D}$. We consider the minimization problem
\[
\mathcal{W}(y)=\int_\Omega W(\nabla y(x))\,dx
\]
for $y \in W^{1,p}_{y_0,D}(\Omega)$ where for a generic function $g$ we write
\[
W^{1,p}_{g,D}(\Omega) = \left\{y\in W^{1,p}(\Omega)\,:\, y = g \mbox{ on }\Gamma_D\right\},
\]
in the sense of trace. We thus interpret $\Gamma_D$ as the Dirichlet part of the boundary, and $\Gamma_N$ as the Neumann boundary. Moreover, for a unit vector $n$, we define the half ball
\[
B_{n}^- := \left\{x\in \R^d\,:\, |x|<1,\,x\cdot n < 0\right\}.
\]

\begin{corollary}\label{cor:unique}
Assume that $W\in C^2(\R^{n\times d})$ satisfies (H3), (H4) and let $\bar{y}\in C^1(\overline{\Omega})\cap W^{1,p}_{y_0,D}(\Omega)$ such that the following conditions hold:
\begin{itemize}
\item $\bar y$ is a weak solution of the Euler-Lagrange equations, ${\rm div} DW(\nabla \bar y) = 0$, i.e.
\[
\int_\Omega DW(\nabla\bar{y}(x))\nabla\varphi(x) dx=0,
\]
for all $\varphi\in  C^1(\Omega)\cap W^{1,p}_{0,D}(\Omega)$;
\item the second variation is strongly positive at $\bar y$, i.e.
\[
\int_\Omega D^2W(\nabla\bar{y}(x))\nabla\varphi(x)\cdot\nabla\varphi(x)dx\geq c\int_\Omega |\nabla\varphi(x)|^2dx,
\]
for all $\varphi\in  C^1(\Omega)\cap W^{1,p}_{0,D}(\Omega)$;
\item $W$ is strongly quasiconvex at $\nabla\bar y(x_0)$ for all $x_0\in \overline{\Omega}$, i.e. 
\[
\int_B \left[W(\nabla\bar y(x_0) + \nabla\varphi(x)) - W(\nabla\bar y(x_0))\right] dx \geq c_0 \int_B |V(\nabla\varphi(x))|^2dx,
\]
for all $\varphi\in W^{1,p}_0(B)$, where $B$ denotes the unit ball in $\R^d$;
\item $W$ is strongly quasiconvex at $\nabla\bar y(x_0)$ for all $x_0\in \Gamma_N$, i.e. denoting by $n(x_0)$ the outward pointing unit normal at $x_0\in \Gamma_N$,
\[
\int_{B^-_{n(x_0)}} W(\nabla\bar y(x_0) + \nabla\varphi(x)|\nabla\bar y(x_0)) dx \geq c_0 \int_{B^-_{n(x_0)}} |V(\nabla\varphi(x))|^2dx,
\]
for all $\varphi\in  W^{1,p}(B^-_{n(x_0)})$ such that $\varphi = 0$ on $\partial B \cap \overline{B^-_{n(x_0)}}$.
\end{itemize}
Then, there exists $\eps_0>0$ and $C>0$ such that 
\begin{align*}
    \mathcal{W}[y] - \mathcal{W}[\bar y]\geq C\int_\Omega|V(\nabla y(x)-\nabla\bar{y}(x))|^2dx,
\end{align*}
for all $y \in  W^{1,p}_{y_0,D}(\Omega)$ with $\|y - \bar{y}\|_{L^{p}(\Omega)}\leq \varepsilon_0$.
\end{corollary}

\begin{proof}
The proof that (H3), (H4), the strong positivity of the second variation and the quasiconvexity conditions imply that
\begin{equation}\label{eq:unique1}
   \int_\Omega W(\nabla\bar y + \nabla\varphi|\nabla\bar y)\geq 0,
\end{equation}
is given in \cite{JudithR,JudithKostas}. Note that the proof relies on proving Proposition \ref{prop:2} also for points on $\Gamma_N$ and appropriate test functions, using the quasiconvexity at the boundary. This is the content of \cite[Proposition 4.6]{JudithKostas} where, due to the presence of lower order terms, $L^\infty$ assumptions are needed which are not required here. Proposition \ref{prop:2} replaces the quasiconvexity conditions for the rest of the proof which thus remains the same. Then, the satisfaction of the Euler-Lagrange equations implies that \eqref{eq:unique1} gives the minimality of $\bar{y}$.

Thus, in order to obtain the lower bound and the quantitative estimate of uniqueness, we prove \eqref{eq:unique1} for the function $\tW$ of Lemma \ref{lemma:tWall}, in place of $W$. In particular, we need to find a constant $c_2=c_2(W,\|\bar y\|_{C^1})$ such that $\tW$
satisfies (H3), (H4), the strong positivity of the second variation, as well as the quasiconvexity conditions. That $c_2$ can be chosen so that (H3), (H4) and the strong quasiconvexity holds is the content of Lemma \ref{lemma:tWall}. That the second variation is strongly positive is part of the proof of Theorem \ref{theorem:2} and we are thus left to infer the quasiconvexity at the boundary. Denoting by $f(\lambda) = |V(\lambda)|^2$, we compute
\begin{align*}
  &  \int_{B^-_{n(x_0)}} \tW(\nabla\bar y(x_0) + \nabla\varphi|\nabla\bar y(x_0))  = \int_{B^-_{n(x_0)}} W(\nabla\bar y(x_0) + \nabla\varphi|\nabla\bar y(x_0)) \\
    &\quad  \quad - c_2 \int_{B^-_{n(x_0)}} f(\nabla\bar y(x_0) + \nabla\varphi(x)|\nabla\bar y(x_0)) \geq (c_0 - c_2 C) \int_{B^-_{n(x_0)}} |V(\nabla\varphi)|^2,
\end{align*}
by the strong quasiconvexity at the boundary and Lemma \ref{lemma:technical_general} (a). We may thus choose $c_2$ small enough depending on $C=C(W,\|\bar y\|_{C^1})$ and $c_0$ to complete the proof.
\end{proof}

\section*{Appendix: Proof of Lemma \ref{lemma:technical_general}}

\begin{proof}
For (a), note that if $|z_1| + |z_2| \leq 1$, we find that
\begin{align*}
& |f(\lm+z_1|\lm) - f(\lm+z_2|\lm)| \leq \int_0^1 \left| \left[D^2f(\lm + s z_1) - D^2f(\lm + s z_2)\right]z_1\cdot z_2 \right| \\
& + \int_0^1 \left| D^2f(\lm + s z_1)z_1\cdot (z_1 - z_2) \right| + \int_0^1 \left| D^2f(\lm + s z_2)z_2\cdot (z_1 - z_2) \right|\\
& \leq C\left( |z_1||z_1 - z_2| + |z_1 - z_2| |z_1| |z_2| + |z_2||z_1 - z_2|\right),
\end{align*}
where $C = C(f,K)$. Since, for $|z_1| + |z_2| \leq 1$, it holds that $ |z_1| |z_2| \leq |z_1| + |z_2| $, we find that for all $|\lm|\leq K$,
\[
|f(\lm+z_1|\lm) - f(\lm+z_2|\lm)|  \leq C (|z_1| + |z_2|)|z_1 - z_2|.
\]
On the other hand, if $|z_1| + |z_2|>1$, we compute that for $|\lm|\leq K$
\begin{align*}
|f(\lm+z_1|\lm) - f(\lm+z_2|\lm)| & \leq |f(\lm + z_1) - f(\lm + z_2)| + |Df(\lm)\cdot (z_1 - z_2)|\\
& \leq C(K) (1 + |z_1|^{p-1} + |z_2|^{p-1})|z_1 - z_2|\\
&\leq C(K)( |z_1| + |z_2| + |z_1|^{p-1} + |z_2|^{p-1})|z_1 - z_2|,
\end{align*}
since $|z_1| + |z_2| > 1$. This completes the proof of (a).

Concerning (b), we follow the same strategy. If $|z|\leq 1$, 
\begin{align*}
|f(\lm_1+z|\lm_1) - f(\lm_2+z|\lm_2)| & \leq \int_0^1 \left|D^2f(\lm_1 + sz) - D^2f(\lm_2 + s z) \right| |z|^2\,ds\\
& \leq C(f,K) \left| \lm_1-\lm_2\right||z|^2,
\end{align*}
hence, given $\delta>0$ we may choose $R \leq  \delta/C(f,K)$. Similarly, for $|z|>1$,
\begin{align*}
|f(\lm_1+z|\lm_1) - f(\lm_2+z|\lm_2)| & \leq \left|f(\lm_1 + z) - f(\lm_2 + z) \right| + \left|f(\lm_1) - f(\lm_2) \right| \\
&\quad +  \left|Df(\lm_1) - Df(\lm_2) \right||z|\\
& \leq C(f,K) (1 + |z| + |z|^{p-1}) \left|\lm_1  -\lm_2 \right|\\
&\leq C(f,K) |\lm_1- \lm_2| |V(z)|^2.
\end{align*}
Hence, $R$ as above suffices to complete the proof of (b).

Regarding (c), we follow \cite{GM09}. For any $|z| \leq 1$, we find $C=C(f,K)>0$ such that
\begin{align*}
f(\lm+z|\lm) = \int_0^1(1-s) D^2 f(\lambda + s z)\,ds\, z \cdot z
\geq - C |z|^2 \geq |z|^p - (C+1)|z|^2.
\end{align*}
On the other hand, if $|z|> 1$, by coercivity, we get
\[
f(\lm+z|\lm) \geq d_1|z|^p - d_2(f,K) - d_3(f,K)|z| \geq d_1|z|^p - (d_2+d_3)|z|^2,
\]
concluding the proof of (c).

For the proof of (d), note that by Young's inequality
\begin{align*}
f(\lm+z|\lm) \geq c|z|^p - C(f,K) - C(\delta)| |Df(\lm)|^q - \delta c|z|^p
 \geq \tilde{c}|z|^p - C(f,K,\delta),
\end{align*}
for $\delta$ small enough. Hence, if $|z|^p \geq 2C(f,K.\delta)/\tilde{c} + 1:= R^p$, we deduce that
\[
f(\lm+z|\lm) \geq \frac{\tilde{c}}{2}|z|^p \geq \frac{\tilde{c}}{4}|V(z)|^2,
\]
as $|z|\geq 1$. On the other hand, for $|z| < R$, 
by strong convexity,
\begin{align*}
f(\lm+z|\lm) & = \int_0^1(1-t) D^2f(\lm+sz)z\cdot z\,ds \\
& \geq \frac{1}{4}\gamma |z|^2 + \frac{R^2}{4}\gamma \frac{|z|^2}{R^2} \geq  \frac{1}{4}\gamma |z|^2 + \frac{R^2}{4R^p}\gamma |z|^p\geq \tilde{c} |V(z)|^2.
\end{align*}
Combining the two cases, we infer the result.
\end{proof}


\end{document}